\documentclass[11pt, a4paper, final]{amsart} 

\usepackage[left = 2.3cm, right = 2.30cm, top = 2.5cm, bottom = 2.5cm]{geometry}
\usepackage[utf8]{inputenc}
\usepackage[english]{babel}
\usepackage[T1]{fontenc}
\usepackage{lmodern}
\usepackage{amsmath, amsfonts, amssymb,  mathabx, mathrsfs, mathtools, dsfont}
\usepackage{graphicx}
\usepackage{subcaption}
\usepackage{hyperref}
\hypersetup{%
colorlinks=true,
linkcolor= red,	
citecolor= red,	
}
\usepackage{aliascnt}
\usepackage[capitalize]{cleveref}
\usepackage{enumitem}

\newtheorem{theorem}{Theorem}[section]

\newaliascnt{lemma}{theorem}
\newtheorem{lemma}[lemma]{Lemma}
\aliascntresetthe{lemma}

\newaliascnt{proposition}{theorem}

\aliascntresetthe{proposition}

\newtheorem{definition}{Definition}[section]
\newaliascnt{algo}{definition}

\aliascntresetthe{algo}
\numberwithin{equation}{section}
\newcommand{\tg}{\tilde{g}}
\newcommand{\tu}{\tilde{u}}
\newcommand{\tv}{\tilde{v}}
\newcommand{\tc}{\tilde{c}}
\newcommand{\la}{\langle}
\newcommand{\ra}{\rangle}
\newcommand{\Gr}{\boldsymbol{\mathrm{Gr}}}
\newcommand{\dd}{\,\mathrm{d}}
\newcommand{\mI}{\mathcal{I}}
\newcommand{\mC}{\mathcal{C}}
\newcommand{\RR}{\mathbb{R}}
\newcommand{\NN}{\mathbb{N}}
\newcommand{\bu}{\bar{u}}
\title{Reduced dynamics for models of pattern formation}
\date{\today}
\author[E.\ Hausenblas]{Erika Hausenblas}
\address{Chair of Applied Mathematics, Montanuniversität Leoben, Austria}
\email{erika.hausenblas@unileoben.ac.at}
\author[T.A.\ Randrianasolo]{Tsiry Avisoa Randrianasolo}
\address{Faculty of Mathematics, Bielefeld University, Germany}
\email{trandria@math.uni-bielefeld.de}
\begin{document}
%
%
%
\begin{abstract}
	The goal of this work is to analyze the long-term behavior of reaction-diffusion systems arising in two-species chemical models and to identify finite sets of modes that determine their asymptotic dynamics. The models considered include, as particular cases the Gray--Scott and the Glycolysis models. These systems are described by coupled reaction-diffusion equations and admit a finite-dimensional representation based on a limited number of spatial Fourier modes that capture their essential reduced dynamics. The concept of determining modes, introduced in this context, is closely related to other approaches that seek finite-dimensional representations of infinite-dimensional dynamics, such as the Proper Orthogonal Decomposition and the construction of Approximate Inertial Manifolds. We prove that the dynamics of the system can be completely characterized by a finite number of low modes, since all higher modes are asymptotically determined by them, thus providing an analytical foundation for reduced dynamics in models of pattern formation.
\end{abstract}
%
%
%
%

\maketitle

%
%
\section{Introduction}\label{sec:intro}
Pattern formation arises from the interactions among different components, potentially influenced by their environment. Alan Turing, a cryptographer and pioneer in computer science, developed algorithms to describe complex patterns emerging from simple inputs and random fluctuations. In his seminal 1952 paper \cite{turing}, Turing proposed that the interaction between two biochemical substances with different diffusion rates could generate patterns. His work addressed a key question in morphogenesis: how a single egg can develop into a complex organism. The mathematical model describes an activator protein that stimulates both its own production and the production of an inhibitory protein whose only function is to suppress the activator. It has been observed that a stable homogeneous pattern can become unstable if the inhibitor diffuses faster than the activator. The interaction between the concentrations of these proteins leads to pattern formation, with their spatiotemporal evolution governed by coupled reaction-diffusion systems, known as the activator-inhibitor model. This phenomenon is referred to as diffusion-driven instability, or {\sl Turing instability}. The fundamental phenomenon in activator-inhibitor systems is that a small deviation from spatial homogeneity can trigger strong positive feedback, amplifying the deviation further. Nonlinearities in the local dynamics, such as those introduced by the inhibitor concentration, can then saturate the Turing instability, resulting in a stable and spatially inhomogeneous pattern. For a more detailed discussion, refer to the recent works by Upadhyay and Iyengar \cite{UI_pattern}, Wei and Winter \cite{wei_pattern}, Keener \cite{keener}, and Perthame \cite{perthame}.

When analyzing these patterns in spatial Fourier modes, it is observed that typically only a finite number of modes play significant roles, while higher modes can often be neglected due to their dissipation. This raises the following questions: how many modes are necessary to effectively analyze and characterize the system’s dynamical behavior? Is it essential to approximate the solution with high fidelity, or is it sufficient to model the system using only a few modes?

The concept of determining modes becomes particularly relevant in this context. Originally introduced in the study of fluid dynamics, the notion of determining modes is crucial for understanding complex phenomena like turbulence, which is characterized by chaotic and unpredictable fluid motion. Identifying and analyzing the underlying structures, or modes, within turbulent flows is essential for accurately predicting the behavior of fluid systems.
Determining modes were introduced to identify the parameters that control turbulent flow; see, for example, the monograph by Foia{\c{s}} \cite[Chapter III]{foias2001navier} and the articles by Jones and Titi \cite{jones1993upper}; and Foia{\c{s}} and Temam \cite{foias1984determination}. Modes are said to be ``determining'' if the high modes are asymptotically controlled by these lower modes. By identifying these determining modes, one can characterize the dominant patterns or structures within the flow. In the long run, the entire system can be described by a finite number of parameters, namely the coefficients of the determining modes.

The concept of determining modes has been successfully applied in numerical simulations to reduce computational time. It is also used in data assimilation \cite{olson2003determining,dataass1} and control \cite{control01}. In the context of reaction-diffusion equations, determining modes have been discussed in \cite{det_reac_diff}, where the systems considered are of a monotone type.

Let $ D \coloneqq [0,1]\times[0,1]$.
In this paper, we are interested in the behavior in the limit $t\to\infty$ of a two-species model of the form
\begin{equation} \label{eq:gs}
\left\{
\begin{aligned}
	\partial_t{u}(t) &= d_1\Delta u(t) +a_1u(t) +b_1v(t) -\gamma u(t)v^2(t) + g_1(t), &\mbox{ in } D,
	\\
	\partial_t{v}(t) &= d_2\Delta v(t) +a_2u(t)+ b_2v(t) + \gamma u(t)v^2(t) + g_2(t), &\mbox{ in } D,
	\\
	u(0) &= u_0\ge 0,\;v(0) = v_0\ge 0, &\mbox{ in } D,
\end{aligned}
\right.
\end{equation}
where for $j=1,2$,
$d_j,a_j,b_j,\gamma$ are constants and  $g_j$ are time-dependent maps to be specified later.
We will consider the problem associated with zero Neumann boundary conditions  on $\partial D$.

Regarding the solution theory, we can refer to \cite{pierre2010global}, where the author discusses questions related to  global solutions to a class of reaction-diffusion systems that satisfies two major properties: the quasi-positivity (P) and the mass-control (M). With these two properties valid, and for certain parameters, it is possible to show the global existence of the classical solution. Moreover, the solution remains positive for all time. That applies, for instance, to the Brusselator model \cite{glansdorff1971structure}, which is a particular case of \eqref{eq:gs}, where
$ a_1 = a_2= 0$; $b_1>0, b_2=-(b_1 + 1)$; $\gamma=1$; $g_1 = 0, g_2>0$.
The Brusselator satisfies the properties (P) + (M), and it was shown, see \cite[Proposition 2]{hollis1987global}, \cite[Theorem 1, Page 140]{rothe2006global}, that it has a unique global classical solution. This result applies also to classical reaction-diffusion systems such as the Gray--Scott, and Glycolysis models \cite{ashkenazi1978spatial,segel1980mathematical}.

The system \eqref{eq:gs} fits also within the settings considered by McGough and Riley in \cite{mcgough2005a}, where they develop a general class of reaction-diffusion problems characterized by telescoping nonlinearities.
The term ``telescoping nonlinearities'' refers to the structure in which nonlinear interaction terms cancel or combine across equations, allowing global bounds (such as mass control) to be derived from the sum of the system components.

This property can also be understood through stoichiometric bookkeeping of the chemical reaction. In a spatially distributed reactor model, each species $i$ satisfies an equation of the form
\begin{equation*}
	\partial_tu_i = d_i\Delta u_i +  \sum_r \nu_{i,r}R_{r}(u) + F(u_i^{\mbox{in}} - u_i),
\end{equation*}
where $\nu_{i,r}$ denotes the stoichiometric coefficient of species $i$ in reaction $r$, and $R_{r}(u)$ is the corresponding reaction rate. The term  $F(u_i^{ \mbox{in}} - u_i)$ represents inflow-outflow effects
at dilution rate $F$, where $u_i^{ \mbox{in}}$ is the concentration of species $i$ in the inflow, and the term accounts for both feeding of fresh material and removal of existing material.

This representation reflects the underlying stoichiometric structure of the chemical reaction. A key property is that, for each reaction $r$, the coefficients satisfy a balance relation of the form $\sum_i \nu_{i,r} = 0$, or more generally $\sum_i \alpha_i\nu_{i,r} = 0$, corresponding to conserved quantities of the system. As a consequence, when summing the equations over all species, the reaction terms cancel.

This cancellation mechanism is precisely what is referred to as telescoping nonlinearities: nonlinear interaction terms appear with opposite signs across the system due to the stoichiometric structure, and therefore vanish in the total balance,
\begin{align*}
	\partial_t \big(\sum_i u_i\big) 
	&= \big(\sum_id_i\Delta u_i\big) 
	+  \sum_r \big(\sum_i\nu_{i,r}\big)R_{r}(u) 
	+  F\big(\sum_iu_i^{ \mbox{in}} - \sum_iu_i\big)
	\\
	&= \big(\sum_id_i\Delta u_i\big)   
	+  F\big(\sum_iu_i^{ \mbox{in}} - \sum_iu_i\big).
\end{align*}

For example, the Gray--Scott model has two reactions ($r = 2$). The first reaction of the form 
$
U + 2V\rightarrow 3V
$
has a rate $R_1(u,v) =k_1 uv^2$. That's a consequence of the law of mass action, which for a reaction of the form
$
aU + bV\rightarrow W
$
gives a reaction rate $ku^av^b$.	
Thus, the second reaction of the form
$
V\rightarrow P
$
has a rate $R_2(u,v) =k_2 v$.
In the classical Gray--Scott formulation, the product $P$ is not included as a dynamical variable, as it is assumed to be chemically inert (waste) and does not feed back into the evolution of the other species $U$ and $V$. Instead, it is removed from the system, allowing the model to be written solely in terms of the active species $U$ and $V$.

Therefore, the reaction contribution for each species is given by
\[
\sum_r \nu_{u,r}R_r(u) 
= -k_1uv^2,\quad \sum_r \nu_{v,r}R_r(u) 
= k_1uv^2 - k_2v,\quad \sum_r \nu_{p,r}R_r(u) 
=  k_2v.
\]
We can identify the stoichiometric coefficients  
$$
(\nu_{u,1}, \nu_{v,1}, \nu_{p,1}) 
= (-1,+1, 0), \quad (\nu_{u,2}, \nu_{v,2}, \nu_{p,2}) 
= (0,-1, +1).
$$
Hence, by summing we can see how the nonlinear reaction terms cancel in the total balance.

These problems have special properties, which for the case of the two-species model \eqref{eq:gs} are given as follows:
\begin{enumerate}[label=($P_\arabic*)$]
\item{\label{P1}} $d_1,d_2 >0$,

\item{\label{P2}} $g_1,g_2\in L^\infty(0,\infty;L^p(D))$, for any $p\ge 2$,

\item{\label{P3}} $\gamma \ge 0$,

\item{\label{P4}} For all $u,v\ge 0$,
$
b_1v + g_1\ge 0 \mbox{ and } a_2u + g_2\ge 0,
$

\item{\label{P5}} There exist constants $c_1\ge 0$ and $c_2 >0$ such that $a_1u + b_1v -\gamma uv^2 + g_1\le c_1(1-c_2u)$,

\item{\label{P6}} $a_1 + a_2\le 0$ and $b_1 + b_2\le 0$.
\end{enumerate}
If the properties \ref{P1}-\ref{P6} hold, McGough and Riley showed several theoretical results, including that all positive solutions to the system remain bounded and smooth for all time. This result holds for a domain $D\subset\RR^N$, $N=2,3$, with a piecewise smooth boundary.  However, we establish all results  in two spatial dimensions ($N = 2$).

The goal of this paper is to show that the asymptotic behavior of the system of partial differential equations given by \eqref{eq:gs} can be effectively monitored by projecting the system onto a finite-dimensional space, known as determining modes. For a precise definition we refer to \cref{defn:determining_modes}. 

The results presented here are closely connected to a recent work on a continuous data assimilation algorithm for the Gray--Scott model~\cite{randrianasolo2025discrete}, where synchronization of a nudged solution with the true dynamics was achieved from coarse cell-averaged measurements. In that work, feedback is applied only to coarse finite volume grids through an interpolation operator $\mI_H$. The present determining modes result provides an alternative characterization of such synchronization: \cref{thm:lambda_main} shows that, once the low modes, corresponding to the observed scales, are controlled, the remaining unobserved high modes decay asymptotically. In this sense, the determining subspace in the current work coincides with the set of spatial modes on which feedback or observation must act to ensure recovery of the full Gray--Scott dynamics. The two results, therefore, describe complementary aspects of the same mechanism: one from the perspective of dynamical systems theory, the other, more algorithmic, from the perspective of control and data assimilation.

Beyond data assimilation, the present work is also connected to other approaches that seek finite-dimensional representations of infinite-dimensional dynamics. Among the most widely used are the \textit{Proper Orthogonal Decomposition} (POD) and the construction of \textit{Approximate Inertial Manifolds} (AIM). Both methods aim to identify low-dimensional structures that capture the long-term dynamics of dissipative partial differential equations.

In POD, one seeks a finite set of empirical orthogonal modes that optimally represent the energy content of the system; see, for instance, the article by Berkooz, Holmes, and Lumley~\cite{berkooz1993the}, and the later developments summarized in \cite{foias2001navier}. The construction of AIM, introduced by Foia\c{s}, Sell, and Temam~\cite{foias1988inertial}, provides a method in which high modes are expressed as Lipschitz functions of the low modes, thereby yielding a finite-dimensional approximation of the global attractor. In contrast, the determining modes approach establishes explicit analytical conditions ensuring that a finite set of modes uniquely determines the full dynamics.

Related ideas have been further developed in the context of \textit{determining nodes}, \textit{determining volume elements}, and \textit{determining forms}, see, e.g., Foia\c{s}, Jolly, Kravchenko, and Titi~\cite{foias2012a}; Foia\c{s} and Temam~\cite{foias1984determination};
Jones and Titi~\cite{jones1992determining, jones1993upper}; and  the work of Kaper, Wang, and Wang on nonlinear   Ginzburg--Landau equations \cite{hans1998determining}. These works share the same objective: to quantify the finite number of degrees of freedom needed to describe the long-time dynamics of dissipative systems. Our present results extend these ideas to reaction-diffusion systems relevant to pattern formation, including the Gray--Scott and Glycolysis models.

The rest of the paper is organized as follows.
In \cref{sec:prelim2}, we collect background materials and necessary assumptions on the
system and present the main result.
In \cref{sec:preliminaries}, we introduce some necessary lemmas and propositions that serve as
intermediate steps in the proof of the main result.
In \cref{sec:lambda_main}, we prove that the reaction-diffusion system \eqref{eq:gs} admits a  set of determining modes, as defined in \cref{defn:determining_modes}.
In \Cref{sec:num_exp}, we perform numerical experiments in support of the main result by applying a  data assimilation algorithm to the Gray--Scott model.
Finally, in \cref{sec:discussion}, we conclude with a discussion of the implications of the main result for other models of Turing pattern, including its connection to data assimilation for the Gray--Scott system.

%
%
\section{Background materials and the main statement}\label{sec:prelim2}
In this section, we will introduce  standard notation and recall results on the solution theory of systems like \eqref{eq:gs}. Then, we present the definition of determining modes and the main result. The proof of the main result is postponed to \cref{sec:lambda_main}.

We denote by $L^2\coloneqq L^2( D)$, the usual Lebesgue space of squared integrable functions. It is endowed with the scalar product $\la\,\cdot\,,\,\cdot\,\ra$.
For $1\leq p\leq \infty$, we denote the usual Lebesgue spaces by $L^p$, which are endowed with the standard norms denoted by
\[
\Vert u\Vert_{L^p}^p\coloneqq\int_{ D} \vert u(x)\vert^p\dd x,\, 1\leq p< \infty, \quad\mbox{ and }\quad\Vert u\Vert_{L^\infty}\coloneqq\sup_{x\in D} \vert u(x)\vert.
\]

In the paper, $(u,v)$ is understood as the global smooth solution to \eqref{eq:gs}, whose existence is made precise in the following lemmas, see also {\cite[Theorem 10]{mcgough2005a}}.
\begin{lemma}\label{lem:global1}
	We assume that the properties \ref{P1}-\ref{P6} hold for the system \eqref{eq:gs}. If $u_0,v_0\in L^2(D)$ then for the solution to \eqref{eq:gs} it holds that $u(t),v(t)\ge 0$ and $u(t),v(t)\in \mC^2(D)$ for all $t\ge 0$.
\end{lemma}
More generally, the following estimates hold for $(u,v)$, see e.g.\ {\cite[Theorem 2]{mcgough2005a}}.
\begin{lemma}\label{lem:global2}
	If $u_0,v_0\in L^p(D)$, then $u,v\in L^\infty(0,\infty;L^p(D))$ for all $p\ge 2$.
\end{lemma}
\subsection{Determining modes}\label{subsec:determining}
Let $(\tg_1,\tg_2)$ be a deterministic perturbation of $(g_1,g_2)$ such that $(\tg_1,\tg_2)$ and $(g_1,g_2)$ have the same asymptotic behavior for large times, i.e.\
\begin{equation}\label{eq:asymptotic_behavior}
	\limsup_{t\to\infty}  \Vert (g_1-\tg_1)(t)\Vert_{L^2}+ \limsup_{t\to\infty}\Vert (g_2-\tg_2)(t)\Vert_{L^2} = 0.
\end{equation}
To $(\tg_1,\tg_2)$ we associate the pair $(\tu,\tv)$ that satisfies for all $t\ge 0$ to the system
\begin{equation} \label{eq:gs2}
	\left\{
	\begin{aligned}
		\partial_t{\tu}(t) &= d_1\Delta \tu(t) +a_1
		\tu(t) +b_1\tv(t) -\gamma\tu(t)\tv^2(t) + \tg_1(t),&\mbox{ in } D,
		\\
		\partial_t{\tv}(t) &= d_2\Delta \tv(t) +a_2\tu(t)+ b_2\tv(t) + \gamma\tu(t)\tv^2(t) + \tg_2(t), &\mbox{ in } D,
		\\
		\tu(0) &= u_0\ge 0,\;\tv(0) = v_0\ge 0, &\mbox{ in } D,
	\end{aligned}
	\right.
\end{equation}
with zero Neumann boundary conditions on $\partial D$. We assume that the following properties hold for the system \eqref{eq:gs2},
\begin{enumerate}[label=($\tilde{P}_\arabic*)$]
	\item{\label{tP1}} $d_1,d_2 >0$,
	
	\item{\label{tP2}} $\tg_1,
	\tg_2\in L^\infty(0,\infty;L^p(D))$, for any $p\ge 2$,
	
	\item{\label{tP3}} $\gamma\ge 0$,
	
	\item{\label{tP4}} For all $u,v\ge 0$, $b_1v + \tg_1\ge 0 \mbox{ and } a_2u + \tg_2\ge 0,$
	
	\item{\label{tP5}} There exist constants $\tc_1\ge 0$ and $
	\tc_2 >0$ such that $a_1u + b_1v - \gamma uv^2 +  \tg_1\le \tc_1(1-\tc_2u)$,
	
	\item{\label{tP6}} $a_1 + a_2\le 0$ and $b_1 + b_2\le 0$.
\end{enumerate}
In what follows, $(\tu,\tv)$ is understood as the global smooth solution to \eqref{eq:gs2}.   \Cref{lem:global1} and \Cref{lem:global2} are also valid for the system \eqref{eq:gs2} provided that the properties \ref{tP1}-\ref{tP6} hold.

Let $\{\psi_m:m\ge 0 \}$  be an orthonormal basis of $L^2(D)$ formed by eigenfunctions of the Laplacian with homogeneous Neumann boundary condition, 
\[
-\Delta\psi_m = \lambda_m\psi_m\mbox{ in } D, \quad \partial_n\psi_m = 0\mbox{ on }\partial D,\quad m\ge 0,
\]
where $0=\lambda_0<\lambda_1\leq\lambda_2\leq\cdots\leq\lambda_m\leq\cdots, \mbox{ and }\lambda_m\to\infty \mbox{ when } m\to\infty.$
We add the normalized constant eigenfunction $\psi_0\coloneqq 1/\sqrt{\vert D\vert}$, where $\vert D\vert$ denotes the Lebesgue measure of the domain $D$.

Given some $M\in\NN$, we define the finite subspace $H_M\coloneqq\mathrm{span}\big\{\psi_0, \psi_1,\ldots, \psi_{M-1}\big\}$ and denote the orthogonal projection of $L^2(D)$ onto $H_M$ by $P_M$, the natural embedding by $I$, and define $Q_M$ by $Q_M\coloneqq I - P_M$.
Given two indices $M$ and $N$ and the solution $(u,v)$ 	 to \eqref{eq:gs},  we define   its Galerkin approximation
by $(P_Mu,P_Nv)$.

Since $\psi_0\in H_M$, the function $Q_M v$ is orthogonal to constants, hence $\int Q_M v = 0$. Therefore the Poincaré inequality applies,
\begin{equation*}
\Vert Q_M v\Vert_{L^2}^2\leq \frac{1}{\lambda_{M}} \Vert \nabla Q_M v\Vert_{L^2}^2.
\end{equation*}

We introduce the following definition of determining modes:
\begin{definition}\label{defn:determining_modes}
The first modes $H_M\coloneqq\mathrm{span}\big\{\psi_0,\psi_1,\ldots,\psi_{M-1}\big\}$ and $H_N\coloneqq \mathrm{span}\big\{\psi_0,\psi_1,\ldots,\psi_{N-1}\big\}$ are called determining modes for the solution $(u,v)$ to \eqref{eq:gs}, if
\begin{equation*}
\lim_{t\to\infty}\Vert P_{M}(u - \tu)(t)\Vert_{L^2} + \lim_{t\to\infty}\Vert P_{N}(v - \tv)(t)\Vert_{L^2}= 0,
\end{equation*}
implies that
\begin{equation*}
\lim_{t\to\infty}\Vert Q_{M}(u - \tu)(t)\Vert_{L^2}+\lim_{t\to\infty}\Vert Q_{N}(v - \tv)(t)\Vert_{L^2}= 0,
\end{equation*}
where, $(\tu,\tv)$ denotes the solution to \eqref{eq:gs2}.
\end{definition}
\subsection{The main statement}\label{subsec:main_stat}

In addition to properties \ref{P1}-\ref{P6}, we suppose that there exist $c_3>0$ and $c_4\ge 0$ such that
\begin{equation}\label{eq:extra_assumption}
	(a_1 + a_2)u + (b_1 + b_2)v \leq -c_3 (u + v) + c_4,\quad\forall u,v\ge 0.
\end{equation}
Let $ c_D\coloneqq \sqrt{\vert D\vert}$, and define
$$
A_2 \coloneqq  c_D \sum_{j=1}^2 (\vert a_j\vert + \vert b_j\vert ), \quad g^* \coloneqq \limsup_{t \to \infty}  \Vert g_1(t) +g_2(t)  \Vert_{L^2}.
$$
We define the following constants
\[
d \coloneqq \min\{d_1, d_2\},
\qquad
F \coloneqq c_1 + c_4c_D + g^* + A_2^2 + c_1c_D,
\qquad
\Gr \coloneqq \frac{F}{d^2 \lambda_1}.
\]

\begin{theorem}\label{thm:lambda_main}
	We assume that the properties \ref{P1}-\ref{P6} hold for the system \eqref{eq:gs}, and the damping assumption \eqref{eq:extra_assumption} is satisfied.
	If there exist $M,N\in\NN$ such that
	\begin{equation*}
		\frac{\min\{\lambda_{M}, \lambda_{N}\}}{36\lambda_1}
		>  \Gr +   \gamma^2 \bigg[d^2\Big(1 + \frac{d\lambda_1 \sqrt{2}}{c_1c_2}\Big)^4  + \frac{2d^5\lambda_1^3}{c_3^3}(1 + \alpha)^2\bigg]\Gr^4,
	\end{equation*}
	where $\alpha \coloneqq   {(d_1^2+ d_2^2)}/{(2d_1d_2)}$, then the first $(M,N)$ modes are determining for the solution $(u,v)$ to  the system \eqref{eq:gs}.
\end{theorem}
We postpone the proof of  \cref{thm:lambda_main} to \cref{sec:lambda_main}.
In the first place, we show some preliminary estimates, notably \cref{lem:determining}, \cref{lem:form}, and \cref{lem:lambda_etimate1} that will be used as intermediate steps in proving \cref{thm:lambda_main}.

%
%
\section{Preliminaries}\label{sec:preliminaries}
We will start with a Gronwall-type Lemma. 
\begin{lemma}\label{lem:determining}
	Let $\alpha$ be a locally integrable real valued function on $(0,\infty)$, satisfying for some $0<T<\infty$ the following conditions:
	\begin{equation}\label{eq:determining_alpha}
		\gamma\coloneqq\liminf_{t\to\infty}\frac{1}{T}\int_{t}^{t+T}\alpha(s)\dd s >0,\quad
		\Gamma\coloneqq\limsup_{t\to\infty}\frac{1}{T}\int_{t}^{t+T}\alpha^-(s)\dd s <\infty,
	\end{equation}
	where $\alpha^-(t) = \max\{-\alpha(t),0\}$.
	Further let $\beta$ be a real-valued measurable function defined on $(0,\infty)$ such that
	\begin{equation}\label{eq:determining_beta}
		\lim_{t\to \infty}\beta(t) =0.
	\end{equation}
	Suppose that $X$ is an absolutely continuous non-negative function on $(0,\infty)$, with $X_0\coloneqq X(0)$, such that
	\begin{equation*}
		\frac{\dd }{\dd t}X(t) + \alpha(t) X(t)\leq \beta(t),\;\mbox{a.e.\  on } (0,\infty).
	\end{equation*}
	Then $X(t)\to 0 $ as $t\to \infty$.
	Actually, we have
	\begin{equation*}
		X(t)\leq X_0\Gamma'\mathrm{e}^{-(\gamma/2T)(t-0)} + \frac{2\Gamma' T}{\gamma}\sup_{0\leq \tau\leq t}\Vert\beta(\tau)
		\Vert_{L^2},
	\end{equation*}
	with $\Gamma'\coloneqq \mathrm{e}^{\Gamma + 1 + \gamma/2}$.
\end{lemma}
\begin{proof}
	The result follows from Gronwall lemma.
	See \cite[Lemma 4.1]{foias1983asymptotic} and  \cite[Lemma 4.1]{jones1992determining}.
\end{proof}

We introduce the notation
\[
\xi \coloneqq u - \tilde{u}, 
\qquad 
\eta \coloneqq v - \tilde{v}, 
\qquad 
h_1 \coloneqq g_1 - \tilde{g}_1, 
\qquad 
h_2 \coloneqq g_2 - \tilde{g}_2;
\]
and the quantity
\begin{align*}
	K_{M,N}(t)\coloneqq{}  &{-\big\la Q_M\xi(t)\big(a_1-\gamma v^2(t)\big) + Q_N\eta(t)\big(b_1 - \gamma(v(t) + \tv(t))\tu(t)\big),Q_M\xi(t)\big\ra}
	\\
	&-{\big\la Q_M\xi(t)\big(a_2+\gamma v^2(t)\big) + Q_N\eta(t)\big(b_2+\gamma(v(t) + \tv(t))\tu(t)\big), Q_N\eta(t)\big\ra},
\end{align*}
which appears in \Cref{lem:form} and in the proof of \Cref{lem:lambda_etimate1}. We now derive an estimate for $K_{M,N}$.
\begin{lemma}\label{lem:estimate_K}
	Let $A_2\coloneqq c_D\sum_{j=1}^2(\vert a_j\vert + \vert b_j\vert)$. The following estimate holds
	\begin{equation}\label{eq:estimate_nlt}
		\begin{split}
			\frac{\vert K_{M,N}\vert }{\big(\Vert Q_M\xi\Vert_{L^2}+\Vert  Q_N\eta\Vert_{L^2}\big)^2}
			\leq  3 \Big(\gamma\big(\Vert  v\Vert_{L^4}^2 + \Vert  \tv\Vert_{L^4}^2 + \Vert\tu\Vert_{L^4}^2\big) + A_2\Big)\bigg(\frac{\Vert\nabla Q_M\xi\Vert_{L^2}+\Vert\nabla Q_N\eta\Vert_{L^2}}{ \Vert Q_M\xi\Vert_{L^2}+\Vert  Q_N\eta\Vert_{L^2} }\bigg).
		\end{split}
	\end{equation}
\end{lemma}
\begin{proof}
	
	By straightforward calculations, we get
	\begin{align*}
		\vert K_{M,N}\vert 
		\leq{}&  {\Vert  a_1-\gamma v^2\Vert_{L^2}\Vert Q_M\xi\Vert_{L^4}^2  + \Vert b_1- \gamma (v + \tv )\tu\Vert_{L^2}\Vert Q_N\eta\Vert_{L^4}^2}
		\\
		&  +\Big(\Vert a_2+ \gamma v^2\Vert_{L^2}+\Vert  b_2+\gamma(v + \tv)\tu\Vert_{L^2}\Big)\Vert Q_M\xi Q_N\eta \Vert_{L^2}
		\\
		\leq{}& \Big(\gamma\Vert  v\Vert_{L^4}^2  + \gamma\Vert (v + \tv)\tu\Vert_{L^2} + \vert a_1\vert + \vert b_1\vert\Big)\Big({ \Vert Q_M\xi\Vert_{L^4}^2+\Vert Q_N\eta\Vert_{L^4}^2}\Big)
		\\
		&  +\Big(\gamma\Vert  v\Vert_{L^4}^2+\gamma\Vert  (v + \tv)\tu\Vert_{L^2} + \vert a_2\vert + \vert b_2\vert\Big)\Big({\Vert Q_M\xi Q_N\eta\Vert_{L^2}}\Big) 
		\\
		\leq{}& \Big(\gamma\Vert  v\Vert_{L^4}^2  + \gamma\Vert (v + \tv)\tu\Vert_{L^2} + A_2\Big)\Big({ \Vert Q_M\xi\Vert_{L^4}^2+\Vert Q_N\eta\Vert_{L^4}^2} + \Vert Q_M\xi Q_N\eta\Vert_{L^2}\Big),
	\end{align*}
	where $A_2\coloneqq c_D\sum_{j=1}^2(\vert a_j\vert + \vert b_j\vert)$. Next, observe that
	\begin{align*}
		\Vert Q_M\xi\Vert_{L^4}^2+\Vert Q_N\eta\Vert_{L^4}^2
		&\leq \Vert Q_M\xi\Vert_{L^2}\Vert\nabla Q_M\xi\Vert_{L^2}+ \Vert Q_N\eta\Vert_{L^2}\Vert\nabla Q_N\eta\Vert_{L^2}
		\\
		&\leq \Big(\Vert Q_M\xi\Vert_{L^2} + \Vert Q_N\eta\Vert_{L^2}\Big)\Big(\Vert\nabla Q_M\xi\Vert_{L^2}+\Vert\nabla Q_N\eta\Vert_{L^2}\Big),
	\end{align*}
	and
	\begin{align*}
		\Vert Q_M\xi Q_N\eta \Vert_{L^2}
		&\leq  \Vert Q_M\xi \Vert_{L^4}\Vert Q_N\eta \Vert_{L^4}
		\\
		&\leq \frac12\Vert Q_M\xi \Vert_{L^4}^2+\frac12\Vert Q_N\eta \Vert_{L^4}^2
		\\
		&\leq \frac12 \Vert Q_M\xi \Vert_{L^2}\Vert\nabla Q_M\xi \Vert_{L^2}+\frac12  \Vert Q_N\eta \Vert_{L^2}\Vert \nabla Q_N\eta \Vert_{L^2}
		\\
		&\leq \frac12 \Big(\Vert Q_M\xi \Vert_{L^2}+\Vert Q_N\eta \Vert_{L^2}\Big)\Big(\Vert\nabla Q_M\xi \Vert_{L^2}+\Vert \nabla Q_N\eta \Vert_{L^2}\Big).
	\end{align*}
	Combining the previous estimates, we obtain
	\begin{align*}
		\frac{\vert K_{M,N}\vert }{\big(\Vert Q_M\xi\Vert_{L^2}+\Vert  Q_N\eta\Vert_{L^2}\big)^2}
	\leq  \frac32 \Big(\gamma\Vert  v\Vert_{L^4}^2  +\gamma\Vert  (v + \tv)\tu\Vert_{L^2} + A_2\Big)\bigg(\frac{\Vert\nabla Q_M\xi\Vert_{L^2}+\Vert\nabla Q_N\eta\Vert_{L^2}}{ \Vert Q_M\xi\Vert_{L^2}+\Vert  Q_N\eta\Vert_{L^2} }\bigg).
	\end{align*}
	We apply  the H\"older and Young inequalities to the term $\Vert  (v + \tv)\tu\Vert_{L^2}$,
	\begin{align*}
		\Vert  (v + \tv)\tu\Vert_{L^2} 
		\leq \Vert  v + \tv\Vert_{L^4}\Vert \tu\Vert_{L^4}
		\leq \tfrac12\Vert  v + \tv\Vert_{L^4}^2 + \tfrac12\Vert \tu\Vert_{L^4}^2.
	\end{align*}
	Finally, using $\Vert  v + \tv\Vert_{L^4}^2\leq 2(\Vert  v \Vert_{L^4}^2 + \Vert  \tv\Vert_{L^4}^2) $, we obtain
	\begin{align*}
		\frac{\vert K_{M,N}\vert }{\big(\Vert Q_M\xi\Vert_{L^2}+\Vert  Q_N\eta\Vert_{L^2}\big)^2}
		\leq  3 \Big(\gamma\big(\Vert  v\Vert_{L^4}^2 + \Vert  \tv\Vert_{L^4}^2 + \Vert\tu\Vert_{L^4}^2\big) + A_2\Big)\bigg(\frac{\Vert\nabla Q_M\xi\Vert_{L^2}+\Vert\nabla Q_N\eta\Vert_{L^2}}{ \Vert Q_M\xi\Vert_{L^2}+\Vert  Q_N\eta\Vert_{L^2} }\bigg).
	\end{align*}
	
	The proof of \eqref{eq:estimate_nlt} is complete.
\end{proof}
\begin{lemma}\label{lem:form}
	We assume that the properties \ref{P1}-\ref{P6} (resp. \ref{tP1}-\ref{tP6}) hold for the system \eqref{eq:gs} (resp. \eqref{eq:gs2}). 
	If there exist $T>0$ and $M,N\in\NN$ such that
	\begin{equation}\label{eq:form1}
		\begin{split}
			&\liminf_{t\to\infty}
			\Bigg(\frac{1}{T}\int_{t}^{t + T}
			\frac{d\Big(\Vert\nabla Q_M\xi(s)\Vert_{L^2} +  \Vert\nabla Q_N\eta(s)\Vert_{L^2}\Big)^2 +   K_{M,N}(s) }{\Big(\Vert Q_M\xi(s)\Vert_{L^2} + \Vert Q_N\eta(s)\Vert_{L^2}\Big)^2}  \dd s \Bigg)>0
		\end{split}
	\end{equation}
	and
	\begin{equation}\label{eq:form2}
		\lim_{t\to\infty}\;\Vert P_{M}\xi(t)\Vert_{L^2} + \lim_{t\to\infty}\;\Vert P_{N}\eta(t)\Vert_{L^2}= 0,
	\end{equation}
	then it holds that
	\begin{equation*}
		\lim_{t\to\infty}\;\Vert Q_{M}\xi(t)\Vert_{L^2}+\lim_{t\to\infty}\;\Vert Q_{N}\eta(t)\Vert_{L^2}= 0.
	\end{equation*}
\end{lemma}
\begin{proof}
	We fix $M,N\in\NN$, and $t>0$.
	Standard calculations give
	\begin{align}
		\label{eq:main1}
		\partial_t\xi(t,x) - d_1\Delta\xi(t,x)
		&= \gamma\big(\tu(t,x)\tv^2(t,x) - u(t,x)v^2(t,x)\big) + a_1\xi(t,x) + b_1\eta(t,x) + h_1(t,x),
		\\
		\label{eq:main2}
		\partial_t\eta(t,x) - d_2\Delta\eta(t,x)
		&= \gamma\big(u(t,x)v^2(t,x) - \tu(t,x)\tv^2(t,x)\big) + a_2\xi(t,x) + b_2\eta(t,x) + h_2(t,x).
	\end{align}
	Observe the following calculations
	\begin{align*}
		\gamma(uv^2 &- \tu\tv^2)  + a_2\xi + b_2\eta
		\\
		&= (u - \tu)\gamma v^2 + (v + \tv)\tu\gamma(v - \tv) + a_2\xi + b_2\eta
		\\\notag
		&= \xi\gamma v^2 + (v + \tv)\tu\gamma
		\eta + a_2\xi + b_2\eta
		\\\notag
		&=(Q_M\xi)(a_2 + \gamma v^2) + \big(b_2+\gamma(v + \tv)\tu\big)(Q_N\eta) + (P_M\xi)(a_2 + \gamma v^2) + \big(b_2+\gamma(v + \tv)\tu\big)(P_N\eta).
	\end{align*}
	We take the scalar product in $L^2(D)$ with $Q_M\xi$ and $Q_N\eta$ on both sides of \eqref{eq:main1}-\eqref{eq:main2}, respectively. Then, we introduce the quantity
	\begin{align*}
		R_{M,N}(t)\coloneqq{} &\big\la(P_M\xi(t))\big(a_1-\gamma v^2(t)\big) + \big(b_1 - \gamma(v(t) + \tv(t))\tu(t)\big)(P_N\eta(t)), Q_M\xi(t)\big\ra
		\\
		&+ \big\la(P_M\xi(t))\big(a_2+\gamma v^2(t)\big) + \big(b_2 + \gamma(v(t) + \tv(t))\tu(t)\big)(P_N\eta(t)), Q_N\eta(t)\big\ra.
	\end{align*}
	It follows that
	\begin{align}\label{eq:ODE}
		\begin{split}
			&\frac12\frac{\dd}{\dd t}\Big[\Vert Q_M\xi(t)\Vert_{L^2}^2 + \Vert Q_N\eta(t)\Vert_{L^2}^2\Big]+  d_1\Vert\nabla Q_M\xi(t)\Vert_{L^2}^2 +  d_2\Vert\nabla Q_N\eta(t)\Vert_{L^2}^2  +K_{M,N}(t)
			\\
			&=  R_{M,N}(t) + \big\la h_1(t), Q_M\xi(t)\big\ra +\big\la h_2(t), Q_N\eta(t)\big\ra.
		\end{split}
	\end{align}
	
	The rest of the proof relies on \Cref{lem:determining}. Thus, the next step consists in writing the ODE~\eqref{eq:ODE} in the form of
	\begin{equation*}
		\frac{\dd }{\dd t}X(t)  + \alpha(t) X(t) \leq \beta(t),
	\end{equation*}
	where $\alpha$ and $\beta$  are functions to be determined.
	To this end, we derive an estimate for $R_{M,N}$. 
	
	Let $\varepsilon>0$. Applying the Cauchy--Schwarz inequality, and the Young inequality we obtain
	\begin{align*}
		&R_{M,N}(t)
		\le  \Big(\Vert  P_M\xi(t) \Vert_{L^4}\Vert a_1-\gamma v^2(t)\Vert_{L^4} + \Vert P_N\eta(t)\Vert_{L^4}\Vert b_1-\gamma(v(t) + \tv(t))\tu(t)\Vert_{L^4}\Big)\Vert Q_M\xi(t)\Vert_{L^2}
		\\
		& \qquad+  \Big(\Vert  P_M\xi(t) \Vert_{L^4}\Vert a_2+ \gamma v^2(t)\Vert_{L^4} + \Vert P_N\eta(t)\Vert_{L^4}\Vert b_2+\gamma (v(t) + \tv(t))\tu(t)\Vert_{L^4}\Big)\Vert Q_N\eta(t)\Vert_{L^2}
		\\
		&\leq  \frac{1}{\varepsilon}   \Big(\Vert  P_M\xi(t) \Vert_{L^4} + \Vert P_N\eta(t)\Vert_{L^4}\Big)^2\Big(\gamma \big(\Vert  v(t)\Vert_{L^8}^2  +  \Vert (v(t) + \tv(t)) \tu(t)\Vert_{L^4}^2\big) + A_4\Big)^2 
		\\
		&\qquad + \varepsilon\Big(\Vert Q_M\xi(t)\Vert_{L^2}^2 + \Vert Q_N\eta(t)\Vert_{L^2}^2\Big),
	\end{align*}
	where $A_4\coloneqq  c_D^{1/2}\sum_{j=1}^2(\vert a_j\vert+\vert b_j\vert)$. 
	As in the proof of \cref{lem:estimate_K}, we apply  the H\"older and Young inequalities to the term $\Vert  (v + \tv)\tu\Vert_{L^4}$ to obtain
	\begin{equation*}
		\Vert  (v + \tv)\tu\Vert_{L^4} 
		\leq \Vert  v + \tv\Vert_{L^8}  \Vert   \tu\Vert_{L^8} 
		\leq \tfrac12 \Vert  v + \tv\Vert_{L^8}^2 + \tfrac12  \Vert   \tu\Vert_{L^8} ^2.
	\end{equation*}
	Using $\Vert  v + \tv\Vert_{L^8}^2\leq 2(\Vert  v \Vert_{L^8}^2 + \Vert  \tv\Vert_{L^8}^2) $, and absorbing the numerical constants into $C$, we obtain
	\begin{align*}
		R_{M,N}(t) \leq &\frac{C}{\varepsilon}   \Big(\Vert  P_M\xi(t) \Vert_{L^4}^2 + \Vert P_N\eta(t)\Vert_{L^4}^2\Big)\Big(\gamma^2\big(\Vert  v(t)\Vert_{L^8}^4 +  \Vert  \tv(t)\Vert_{L^8}^4  + \Vert \tu(t)\Vert^4_{L^8}\big) + A_4^2\Big) 
		\\
		&\qquad + \varepsilon\Big(\Vert Q_M\xi(t)\Vert_{L^2}^2 + \Vert Q_N\eta(t)\Vert_{L^2}^2\Big).
	\end{align*}

	Let $w = \xi(t)$ or $w = \eta(t)$.  Since $-\Delta \psi_m = \lambda_m \psi_m$ and $\{\psi_m\}$ is an orthonormal basis of $L^2(D)$,
	\[
	\Vert \nabla P_Mw\Vert_{L^2}^2 
	= \sum_{m=0}^{M-1} \lambda_m \vert \la w,\psi_m\ra\vert^2
	\le \lambda_{M-1} \sum_{m=0}^{M-1} \vert \la w,\psi_m\ra\vert^2
	= \lambda_{M-1} \Vert P_Mw\Vert_{L^2}^2.
	\]
	By the Ladyzhenskaya inequality in two dimensions, see e.g. \cite[Proposition III.2.35]{boyer2013mathematical},
	\[
	\Vert P_Mw\Vert_{L^4}^2 \le  \Vert P_Mw\Vert_{L^2}\Vert\nabla P_Mw\Vert_{L^2}
	\le \lambda_{M-1}^{1/2}\Vert P_Mw\Vert_{L^2}^2.
	\]
	Substituting these bounds into the estimate for $R_{M,N}(t)$, we obtain
	\begin{align*}
		R_{M,N}(t) 
		&\leq \frac{C\ell}{\varepsilon}   \Big(\Vert  P_M\xi(t) \Vert_{L^2}^2 + \Vert P_N\eta(t)\Vert_{L^2}^2\Big)\Big(\gamma^2\big(\Vert  v(t)\Vert_{L^8}^4 +  \Vert  \tv(t)\Vert_{L^8}^4  + \Vert \tu(t)\Vert^4_{L^8}\big) + A_4^2\Big) 
		\\
		&+ \varepsilon\big(\Vert Q_M\xi(t)\Vert_{L^2}^2 + \Vert Q_N\eta(t)\Vert_{L^2}^2\big),
	\end{align*}
	where $\ell\coloneqq 6\max\{\lambda_{M-1}^{1/2}, \lambda_{N-1}^{1/2}\}$.

	Now, we insert the estimate for $R_{M,N}$ into \eqref{eq:ODE}. After reorganizing the resulting inequality, the ODE~\eqref{eq:ODE} takes the form required by \Cref{lem:determining} with
	\begin{align*}
		X(t)&\coloneqq \Vert Q_M\xi(t)\Vert_{L^2}^2 + \Vert Q_N\eta(t)\Vert_{L^2}^2,
		\\
		\alpha(t)&\coloneqq \frac{2d\Big(\Vert\nabla Q_M\xi(t)\Vert_{L^2} +  \Vert\nabla Q_N\eta(t)\Vert_{L^2}\Big)^2 +  2 K_{M,N}(t) }{\Big(\Vert Q_M\xi(t)\Vert_{L^2} + \Vert Q_N\eta(t)\Vert_{L^2}\Big)^2}   - 2\varepsilon,
		\\
		\beta(t)&\coloneqq  \frac{C\ell}{\varepsilon} \Big(\Vert  P_M\xi(t) \Vert_{L^2}^2 + \Vert P_N\eta(t)\Vert_{L^2}^2\Big)\Big(\gamma^2\big(\Vert  v(t)\Vert_{L^8}^4 +  \Vert  \tv(t)\Vert_{L^8}^4  + \Vert \tu(t)\Vert^4_{L^8}\big) + A_4^2\Big)  
		\\
		&\qquad + \frac{2}{\varepsilon}\Big(\Vert h_1(t)\Vert_{L^2}^2 +\Vert h_2(t)\Vert_{L^2}^2\Big).
	\end{align*}
	It remains to show that $\beta$ (resp.\,$\alpha$) satisfies to the hypothesis \eqref{eq:determining_beta}   (resp.\,\eqref{eq:determining_alpha}) of \cref{lem:determining}.
	
	The hypothesis \eqref{eq:determining_beta} of \cref{lem:determining} on $\beta$ holds by the asymptotic behavior \eqref{eq:asymptotic_behavior}, \cref{lem:global2}, and the control on the  first $(M,N)$ modes \eqref{eq:form2}.
	The hypothesis \eqref{eq:determining_alpha} of \cref{lem:determining} on $\alpha$ holds as a consequence of \eqref{eq:form1}.
	Indeed,  on one hand, we can fix $\varepsilon >0$ small enough such that
	\begin{equation*}
		\liminf_{t\to\infty}\frac{1}{T}\int_t^{t + T}\alpha(s)\dd s > 0.
	\end{equation*}
	On the other hand, using the estimate \eqref{eq:estimate_nlt} for $K_{M,N}$, we obtain
	\begin{align*}
		\alpha(t)&\ge 2d\bigg(\frac{\Vert\nabla Q_M\xi(t)\Vert_{L^2} +  \Vert\nabla Q_N\eta(t)\Vert_{L^2}}{\Vert Q_M\xi(t)\Vert_{L^2} + \Vert Q_N\eta(t)\Vert_{L^2}}\bigg)^2
		\\
		&\hspace{25pt}-   6 \Big(\gamma\big(\Vert  v\Vert_{L^4}^2 + \Vert  \tv\Vert_{L^4}^2 + \Vert\tu\Vert_{L^4}^2\big) + A_2\Big)\bigg(\frac{\Vert\nabla Q_M\xi(t)\Vert_{L^2}+\Vert\nabla Q_N\eta(t)\Vert_{L^2}}{ \Vert Q_M\xi(t)\Vert_{L^2}+\Vert  Q_N\eta(t)\Vert_{L^2} }\bigg)  -2\varepsilon.
	\end{align*}
	Using Young inequality, we have
	\begin{align*}
		&6\Big(\gamma\big(\Vert  v\Vert_{L^4}^2 + \Vert  \tv\Vert_{L^4}^2 + \Vert\tu\Vert_{L^4}^2\big) + A_2\Big)\bigg(\frac{\Vert\nabla Q_M\xi(t)\Vert_{L^2}+\Vert\nabla Q_N\eta(t)\Vert_{L^2}}{ \Vert Q_M\xi(t)\Vert_{L^2}+\Vert  Q_N\eta(t)\Vert_{L^2} }\bigg) 
		\\
		&\leq \frac{3}{4d}\Big(\gamma\big(\Vert  v\Vert_{L^4}^2 + \Vert  \tv\Vert_{L^4}^2 + \Vert\tu\Vert_{L^4}^2\big) + A_2\Big)^2
		+ \frac{4d}{2}\bigg(\frac{\Vert\nabla Q_M\xi(t)\Vert_{L^2}+\Vert\nabla Q_N\eta(t)\Vert_{L^2}}{ \Vert Q_M\xi(t)\Vert_{L^2}+\Vert  Q_N\eta(t)\Vert_{L^2} }\bigg)^2
		\\
		&\leq \frac{3}{4d}\Big( \gamma^2\big(\Vert  v\Vert_{L^4}^4 + \Vert  \tv\Vert_{L^4}^4+ \Vert  \tu\Vert_{L^4}^4\big) +A_2^2\Big)
		+ {2d}\bigg(\frac{\Vert\nabla Q_M\xi(t)\Vert_{L^2}+\Vert\nabla Q_N\eta(t)\Vert_{L^2}}{ \Vert Q_M\xi(t)\Vert_{L^2}+\Vert  Q_N\eta(t)\Vert_{L^2} }\bigg)^2.
	\end{align*}
	It implies that
	\begin{equation*}
		\alpha(t) \ge - {\frac{3}{4d}\Big( \gamma^2\big(\Vert  v\Vert_{L^4}^4 + \Vert  \tv\Vert_{L^4}^4+ \Vert  \tu\Vert_{L^4}^4\big) +A_2^2\Big)}  -2\varepsilon,
	\end{equation*}
	and thus
	\begin{equation*}
		\alpha^-(t)\leq  {\frac{3}{4d}\Big(  \gamma^2\big(\Vert  v\Vert_{L^4}^4 + \Vert  \tv\Vert_{L^4}^4+ \Vert  \tu\Vert_{L^4}^4\big) +A_2^2\Big)}+ 2\varepsilon,
	\end{equation*}
	which, by \cref{lem:global2}, provides \eqref{eq:determining_alpha}.
	Finally, every assumption of \cref{lem:determining} is satisfied, we can conclude that
	\begin{equation*}
		\lim_{t\to \infty}\Vert Q_M\xi(t)\Vert_{L^2}^2 + \lim_{t\to \infty}\Vert Q_N\eta(t)\Vert_{L^2}^2 =0.
	\end{equation*}
	This completes the proof of \cref{lem:form}.	
\end{proof}

\begin{lemma}\label{lem:lambda_etimate1}
Recall that $A_2\coloneqq c_D\sum_{j=1}^2(\vert a_j\vert + \vert b_j\vert)$. If there exist $T>0$ and $M,N\in\NN$ such that
\begin{equation}\label{eq:lambda_etimate1}
\frac{A_2^2}{d^2}+\frac{ \gamma^2}{d^2   }\bigg(\limsup_{t\to\infty} \frac{1}{T}\int_t^{t+T}\Big(\Vert  v(s)\Vert_{L^4}^4 + \Vert  \tv(s)\Vert_{L^4}^4 + \Vert\tu(s)\Vert_{L^4}^4\Big)\dd s \bigg) <  \frac1{36}\min\{\lambda_{M}, \lambda_{N}\},
\end{equation}
then the first $(M,N)$ modes are determining for the solution $(u,v)$ to  the system \eqref{eq:gs}.
\end{lemma}
\begin{proof}
	We fix $T,t>0$ and $M,N\in\NN$.
	We introduce the quantity
	\begin{equation*}
		H_{M,N}( \xi(s),\eta(s))\coloneqq \frac{\Vert\nabla Q_M\xi(s)\Vert_{L^2} +  \Vert\nabla Q_N\eta(s)\Vert_{L^2}}{\Vert Q_M\xi(s)\Vert_{L^2} + \Vert Q_N\eta(s)\Vert_{L^2}}.
	\end{equation*}

	Using \Cref{lem:estimate_K}, we have
	\begin{align*}
		&\frac{1}{T}\int_{t}^{t + T}\bigg(2d\Big(H_{M,N}( \xi(s),  \eta(s))\Big)^2 + \frac{2 K_{M,N}(s)}{\big(\Vert Q_M\xi(s)\Vert_{L^2} + \Vert Q_N\eta(s)\Vert_{L^2}\big)^2} \bigg)\dd s
		\\
		&\ge \frac{2d}{T}\int_{t}^{t + T} \Big(H_{M,N}( \xi(s),  \eta(s))\Big)^2  \dd s 
		\\
		&\quad-  \frac{6}{T}\int_{t}^{t + T}  \Big(\gamma\big(\Vert  v(s)\Vert_{L^4}^2 + \Vert  \tv(s)\Vert_{L^4}^2 + \Vert\tu(s)\Vert_{L^4}^2\big) + A_2\Big)\Big(H_{M,N}( \xi(s),  \eta(s))\Big)\dd s.
	\end{align*}
	Using the Cauchy--Schwarz inequality, we obtain
	\begin{align*}
		&\frac{1 }{T}\int_{t}^{t + T}  \Big(\gamma\big(\Vert  v(s)\Vert_{L^4}^2 + \Vert  \tv(s)\Vert_{L^4}^2 + \Vert\tu(s)\Vert_{L^4}^2\big) + A_2\Big)\Big(H_{M,N}( \xi(s),  \eta(s))\Big)\dd s
		\\
		&\leq \bigg(\tfrac{1}{T}\int_t^{t+T}\Big(\gamma\big(\Vert  v(s)\Vert_{L^4}^2 + \Vert  \tv(s)\Vert_{L^4}^2 + \Vert\tu(s)\Vert_{L^4}^2\big) + A_2\Big)^2\dd s\bigg)^{\frac12} \bigg(\tfrac{1}{T}\int_{t}^{t + T} \Big(H_{M,N}( \xi(s),  \eta(s))\Big)^2\dd s\bigg)^{\frac12}
		\\
		&\leq \bigg(\tfrac{4}{T}\int_t^{t+T}\Big(\gamma^2\big(\Vert  v(s)\Vert_{L^4}^4 + \Vert  \tv(s)\Vert_{L^4}^4 + \Vert\tu(s)\Vert_{L^4}^4\big) + A_2^2\Big)\dd s\bigg)^{\frac12} \bigg(\tfrac{1}{T}\int_{t}^{t + T} \Big(H_{M,N}( \xi(s),  \eta(s))\Big)^2\dd s\bigg)^{\frac12}.
	\end{align*}
	Also, recall that since $\psi_0\in H_M,H_N$, the functions $Q_M u$ and $Q_N v$ are orthogonal to constants, hence $\int Q_M \xi = \int Q_N \eta = 0$. Thus, by the Poincaré inequality we have
	\begin{equation*}
		\Big(H_{M,N}( \xi(s),\eta(s))\Big)^2\ge  \min\{\lambda_{M}, \lambda_{N}\}.
	\end{equation*}
	It follows that
	\begin{align*}
		&\frac{1}{T}\int_{t}^{t + T}\bigg(2d\Big(H_{M,N}( \xi(s),  \eta(s))\Big)^2 + \frac{2 K_{M,N}(s)}{\big(\Vert Q_M\xi(s)\Vert_{L^2} + \Vert Q_N\eta(s)\Vert_{L^2}\big)^2} \bigg)\dd s
		\\
		&\ge \bigg(\frac{1}{T}\int_{t}^{t + T} \Big(H_{M,N}( \xi(s),  \eta(s))\Big)^2\dd s\bigg)^{\frac12}
		\times \Bigg[2d\bigg(\frac{1}{T}\int_{t}^{t + T} \Big(H_{M,N}( \xi(s),  \eta(s))\Big)^2  \dd s\bigg)^{\frac12} 
		\\
		&\hspace{150pt}-  6\bigg(\frac{4}{T}\int_t^{t+T}\Big(\gamma^2\big(\Vert  v(s)\Vert_{L^4}^4 + \Vert  \tv(s)\Vert_{L^4}^4 + \Vert\tu(s)\Vert_{L^4}^4\big) + A_2^2\Big)\dd s\bigg)^{\frac12} \Bigg]
		\\
		&{}
		\ge   \bigg(\frac{1}{T}\int_{t}^{t + T} \Big(H_{M,N}( \xi(s),  \eta(s))\Big)^2\dd s\bigg)^{\frac12}\times \delta,
	\end{align*}
	where
	\[
	\delta\coloneqq 2 d\Big(\min\{\lambda_{M}, \lambda_{N}\}\Big)^{\frac12} - 12\bigg(\frac{1}{T}\int_t^{t+T}\Big(\gamma^2\big(\Vert  v(s)\Vert_{L^4}^4 + \Vert  \tv(s)\Vert_{L^4}^4 + \Vert\tu(s)\Vert_{L^4}^4\big) + A_2^2\Big)\dd s\bigg)^{\frac12}.
	\]
	If $T>0$ and $M,N\in\NN$ are such that \eqref{eq:lambda_etimate1} holds, then  eventually $\delta >0$. Consequently,
	\begin{equation*}
			\liminf_{t\to\infty}\Bigg(\frac{1}{T}\int_{t}^{t + T}2d\Big(H_{M,N}( \xi(s),  \eta(s))\Big)^2 + \frac{2 K_{M,N}(s)}{\big(\Vert Q_M\xi(s)\Vert_{L^2} + \Vert Q_N\eta(s)\Vert_{L^2}\big)^2} \dd s\Bigg)>0,
	\end{equation*}
	in which case by \cref{lem:form}, the first $(M,N)$ modes are determining for the system \eqref{eq:gs}.
	
	That completes the proof.
\end{proof}

%
%
\section{Proof of  \texorpdfstring{\cref{thm:lambda_main}}{the main result}}\label{sec:lambda_main}
We split the proof of \cref{thm:lambda_main} into two lemmas, then conclude.
Let $(u,v)$ be any solution to the system \eqref{eq:gs}.
\begin{lemma}\label{lem:lambda_intermediate_1}
For all $t\ge 0$ and $T>0$, it holds that
\begin{equation}\label{eq:lambda_intermediate_1}
\frac{1}{d^2}\bigg(\limsup_{t\to\infty}\frac{1}{T}\int_{t}^{t + T}\Vert   u(s)\Vert^4_{L^4}\dd s\bigg)
\le \Big(\frac{d}{2 T}   + \lambda_1 d ^2\Big)\bigg(1 +\frac{d\lambda_1 \sqrt{2}}{c_1c_2}\bigg)^4\Gr^4.
\end{equation}
\end{lemma}
We note that \Cref{lem:lambda_intermediate_1} applies equally to $\tu$, since $(\tu, \tv)$ satisfies the same structural assumptions~\ref{tP1}-\ref{tP6}. Therefore, the same estimate holds with $u$ replaced by $\tu$.
\begin{proof}
We fix $t\ge 0$ and $T>0$. 
Standard calculations give
\begin{equation*}
	\frac12\frac{\dd}{\dd t} \Vert u(t)\Vert_{L^2}^2 + d_1\Vert\nabla u(t)\Vert^2_{L^2} = \int_D\Big(a_1 u(t,x) + b_1v(t,x) -\gamma u(t,x)v^2(t,x) + g_1(t,x)\Big)u(t,x)\dd x.
\end{equation*}
Then, using the property \ref{P5}, we obtain
\begin{align}\label{eq:intermediate1}
\frac12\frac{\dd}{\dd t} \Vert u(t)\Vert_{L^2}^2 + d_1\Vert\nabla u(t)\Vert^2_{L^2} \le c_1c_D\Vert u(t)\Vert_{L^2}.
\end{align}
Since we consider homogeneous Neumann boundary conditions, we use the Poincaré--Wirtinger inequality, see e.g.~\cite[Proposition III.2.39]{boyer2013mathematical},
\begin{equation*}
\Vert \nabla u(t)\Vert_{L^2}^2\ge \lambda_1 \Big(\Vert u(t)\Vert_{L^2}^2 - c_D^2\vert \bu(t)\vert^2\Big) \quad\mbox{ with }\quad \bu(t)\coloneqq \frac{1}{\vert D\vert}\int u(t).
\end{equation*}
Therefore,
\begin{align*}
\frac12\frac{\dd}{\dd t} \Vert u(t)\Vert_{L^2}^2 
+ d_1\lambda_1 \Vert u(t)\Vert^2_{L^2} 
\le c_1c_D\Vert u(t)\Vert_{L^2} 
+d_1\lambda_1c_D^2\vert \bu(t)\vert^2.
\end{align*}
Using the Young inequality,
\begin{align*}
	\frac{\dd}{\dd t} \Vert u(t)\Vert_{L^2}^2 
	+  d_1\lambda_1 \Vert u(t)\Vert^2_{L^2} 
	\le \frac{(c_1c_D)^2}{ d_1\lambda_1}  
	+2d_1\lambda_1c_D^2\vert \bu(t)\vert^2.
\end{align*}
It follows that
\begin{equation*}
	\limsup_{t \to \infty}\Vert u(t)\Vert^2_{L^2} 
	\leq  \Big(\frac{c_1c_D}{ d_1\lambda_1} \Big)^2 
	+2c_D^2 \limsup_{t \to \infty}\vert \bu(t)\vert^2.
\end{equation*}
Next, we derive an estimate for $\vert \bu(t)\vert$ that appears in the right-hand side of the above inequality.  Integrating the equation for $u$, i.e., the first equation in~\eqref{eq:gs}, over $D$ and using again the property~\ref{P5}, we obtain
\begin{equation*}
\frac{\dd}{\dd t} \bu(t) =\frac{1}{\vert D\vert}\int_D \big( a_1 u + b_1 v - \gamma uv^2 + g_1\big) \leq c_1 - c_1 c_2 \bu(t).
\end{equation*}
Consider the associated ODE, 
$
\frac{\dd}{\dd t} z(t) + c_1c_2 z(t) = c_1,
$
whose solution is given by
$
z(t) = z(0)\mathrm{e}^{-c_1c_2 t} + \frac{1}{c_2}(1 - \mathrm{e}^{-c_1c_2 t}).
$
Since $\bu$ satisfies to $\frac{\dd}{\dd t} \bu(t) + c_1c_2 \bu(t) \leq c_1$, thus, by comparison principle
\[
\bu(t) \le \bu(0)\mathrm{e}^{-c_1c_2 t} + \frac{1}{c_2}(1 - \mathrm{e}{-c_1c_2 t}).
\]
Thus, we arrive at
\begin{equation*}
	\limsup_{t \to \infty}\Vert u(t)\Vert^2_{L^2} 
	\leq  \Big(\frac{c_1c_D}{ d_1\lambda_1} \Big)^2 
	+\frac{2 c_D^2}{c_2^2}.
\end{equation*}
Note that
\begin{equation*}
	\Big(\frac{c_1c_D}{ d_1\lambda_1} \Big)^2 
	+\frac{2 c_D^2}{c_2^2} 
	\leq 
	 \bigg(   \frac{c_1c_D}{ d_1\lambda_1}  + \frac{c_D\sqrt{2}}{c_2} \bigg)^2.
\end{equation*}
On recalling that $\Gr \coloneqq F/(d^2\lambda_1)$, with $F \coloneqq c_1 + c_4c_D + g^* + A_2^2 + c_1c_D$, we rearrange the right-hand side to get
\begin{equation*}
	\limsup_{t \to \infty}\Vert u(t)\Vert^2_{L^2} 
	\leq \bigg(   d\Big(\frac{c_1c_D}{ d^2\lambda_1}\Big)  + \frac{d^2\lambda_1 \sqrt{2}}{c_1c_2}\Big(\frac{c_1c_D}{d^2\lambda_1}\Big) \bigg)^2
	\leq  d^2\bigg(1 +\frac{d\lambda_1 \sqrt{2}}{c_1c_2}\bigg)^2\Gr^2.
\end{equation*}

Returning to~\eqref{eq:intermediate1}. Integrating over $[t, t+ T]$, dividing by $T$, and taking $\limsup_{t \to \infty}$, we obtain
\begin{align*}
\bigg(\limsup_{t\to\infty}\frac{1}{T}\int_{t}^{t + T}\Vert  \nabla u(s)\Vert^2_{L^2}\dd s\bigg)
&\le \frac{d}{2T}\bigg(1 +\frac{d\lambda_1 \sqrt{2}}{c_1c_2}\bigg)^2\Gr^2  + c_1c_D\bigg(1 +\frac{d\lambda_1 \sqrt{2}}{c_1c_2}\bigg)\Gr .
\end{align*}
Again, by rearranging the constants, we obtain
\begin{align}\label{eq:lambda_estimate_1}
	\bigg(\limsup_{t\to\infty}\frac{1}{T}\int_{t}^{t + T}\Vert  \nabla u(s)\Vert^2_{L^2}\dd s\bigg)
	\le \Big(\frac{d}{2T}  + d^2\lambda_1\Big)\bigg(1 +\frac{d\lambda_1 \sqrt{2}}{c_1c_2}\bigg)^2  \Gr^2 .
\end{align}

Finally, by the Ladyzhenskaya inequality in two dimensions, see e.g.~\cite[Proposition III.2.35]{boyer2013mathematical}, 
\begin{equation*}
\Vert   u(s)\Vert^4_{L^{4}}\leq \Vert   u(s)\Vert^{2}_{L^2} \Vert   \nabla u(s)\Vert^{2}_{L^2}.
\end{equation*}
Thus, we have
\begin{align*}
\frac{1}{d^2}\bigg(\limsup_{t\to\infty}\frac{1}{T}\int_{t}^{t + T}\Vert   u(s)\Vert^4_{L^4}\dd s\bigg)
&\leq \frac{1}{d^2}\bigg(\limsup_{t\to\infty}\Vert   u(t)\Vert_{L^2}\bigg)^2\bigg( \limsup_{t\to\infty}\frac{1}{T}\int_{t}^{t + T}\Vert   \nabla u(s)\Vert^2_{L^2}\dd s\bigg)
\\
&\leq  \Big(\frac{d}{2T}  + d^2\lambda_1\Big)\bigg(1 +\frac{d\lambda_1 \sqrt{2}}{c_1c_2}\bigg)^4\Gr^4.
\end{align*}

This establishes \eqref{eq:lambda_intermediate_1} and completes the proof of the lemma.
\end{proof}

\begin{lemma}\label{lem:lambda_intermediate_2}
	For all $t\ge 0$ and $T>0$, it holds that
	\begin{equation}\label{eq:lambda_intermediate_2}
		\frac 1{d^2}\bigg( \limsup_{t \to \infty}\frac{1}{T}\int_{t}^{t + T} \Vert   v(s)\Vert_{L^4}^4 \dd s \bigg)
		\leq d^5\lambda_1^4 \bigg(\frac{c_3 }{(1+\alpha)^2} + \frac{1}{T}\bigg) \bigg(\frac{1+\alpha}{c_3}\bigg)^4\Gr^4,
	\end{equation}
	with $\alpha\coloneqq  {(d_1^2+ d_2^2)}/{(2d_1d_2)}$.
\end{lemma}
The estimate in \cref{lem:lambda_intermediate_2} holds with $v$ replaced by $\tv$.  
\begin{proof}
	We fix $t\ge 0$, and $T>0$.
	Observe that $w\coloneqq u + v$ satisfies to
	\begin{equation*}
		\partial_t w(t,x) = d_1\Delta u(t,x) + d_2\Delta v(t,x) + (a_1 +a_2) u(t,x) + (b_1 + b_2) v(t,x) + (g_1+  g_2)(t,x).
	\end{equation*}
	The additional damping assumption \eqref{eq:extra_assumption} implies that
	\begin{align*}
		&\frac12\frac{\dd}{\dd t}\Vert w(t)\Vert_{L^2}^2 + d_1\la\nabla u(t),\nabla w(t)\ra + d_2\la\nabla v(t),\nabla w(t)\ra + c_3\Vert w(t)\Vert_{L^2}^2
		\\
		&\leq c_4\Vert w(t) \Vert_{L^1} + \Vert (g_1+g_2)(t)\Vert_{L^2}\Vert w(t)\Vert_{L^2}.
	\end{align*}
	Since $w=u+v$, we have
	\begin{equation*}
		d_1\la\nabla u(t),\nabla w(t)\ra + d_2\la\nabla v(t),\nabla w(t)\ra 
		= d_1\Vert \nabla u(t)\Vert_{L^2}^2
		+ d_2\Vert \nabla v(t)\Vert_{L^2}^2
		+ (d_1 + d_2)\la\nabla u(t),\nabla v(t)\ra.
	\end{equation*}
	By Young inequality,
	\begin{equation*}
		(d_1 + d_2) \vert\la\nabla u(t),\nabla v(t)\ra\vert 
		\leq \frac{d_2}{2}\Vert\nabla v(t)\Vert_{L^2}^2 
		+ \frac{(d_1+ d_2)^2}{2d_2}\Vert \nabla u(t)\Vert_{L^2}^2.
	\end{equation*}
	Recall $F \coloneqq c_1 + c_4c_D + g^* + A_2^2 + c_1c_D$. Thus, 
	\begin{align}\label{eq:w}
		\frac12\frac{\dd}{\dd t}\Vert w(t)\Vert_{L^2}^2 
		+ \frac{d_2}2\Vert \nabla v(t)\Vert_{L^2}^2
		+ c_3\Vert w(t)\Vert_{L^2}^2 
		\leq (c_4c_D+ g^*)\Vert w(t)\Vert_{L^2}  
		+ {c_{5}}\Vert \nabla u(t)\Vert_{L^2}^2,
	\end{align}
	with $c_{5}\coloneqq   {(d_1^2+ d_2^2)}/{(2d_2)}$.

	Since $u,v\ge 0$, we have $0\le u \le w = u+ v$ and from \eqref{eq:intermediate1},  
	\begin{align*} 
		\frac12\frac{\dd}{\dd t} \Vert u(t)\Vert_{L^2}^2 + d_1\Vert\nabla u(t)\Vert^2_{L^2} \le c_1c_D\Vert w(t)\Vert_{L^2}.
	\end{align*}
	Set $\alpha\coloneqq c_5/d_1$. Multiplying the previous inequality by $\alpha$, we obtain
	\begin{align*} 
		 c_5\Vert\nabla u(t)\Vert^2_{L^2} \le \alpha c_1c_D\Vert w(t)\Vert_{L^2} - \frac\alpha 2\frac{\dd}{\dd t} \Vert u(t)\Vert_{L^2}^2 .
	\end{align*}
	Substituting this into \eqref{eq:w}, we get
	\begin{equation}\label{eq:w1}
		\frac12\frac{\dd}{\dd t}\Big(\Vert w(t)\Vert_{L^2}^2 + \alpha\Vert u(t)\Vert_{L^2}^2 \Big)
		+ \frac{d_2}2\Vert \nabla v(t)\Vert_{L^2}^2
		+ {c_3}\Vert w(t)\Vert_{L^2}^2 
		\leq F \Vert w(t)\Vert_{L^2}.
	\end{equation}
	Define $y(t)\coloneqq \Vert w(t)\Vert_{L^2}^2 + \alpha\Vert u(t)\Vert_{L^2}^2$. Because $0\leq u \leq w$, we have
	\[
	\Vert w(t)\Vert_{L^2}^2 \leq y(t) \leq (1+\alpha) \Vert w(t)\Vert_{L^2}^2.
	\]
	Thus, dropping the nonnegative term involving $\nabla v$ and using the Young inequality, we obtain
	\begin{equation*}
		 \frac{\dd}{\dd t} y(t) + a y(t) \leq \frac{2F^2}{a},
	\end{equation*}
	where $a \coloneqq 2 c_3/(1+\alpha)$. Multiplying the inequality by $\mathrm{e}^{at}$, we get
	\begin{equation*}
		\frac{\dd}{\dd t}\Big( \mathrm{e}^{at} y(t)\Big) \leq \frac{2F^2}{a}\mathrm{e}^{at}.
	\end{equation*}
	Integrating from 0 to $t$,  we find
	\begin{equation*}
		y(t)\leq y(0)\mathrm{e}^{-at} + \frac{2F^2}{a^2}(1 - \mathrm{e}^{-at}).
	\end{equation*}
	Letting $t\to\infty$ and substitute $a$, we obtain
	\begin{equation*}
		\limsup_{t \to \infty} y(t) \le \frac{2F^2}{a^2} \le \bigg(\frac{1+\alpha}{c_3}\bigg)^2F^2.
	\end{equation*}
	Recall that $\Vert w(t)\Vert_{L^2}^2\leq y(t)$ and $0\leq v\leq w$. Thus, we also get
	\begin{equation*}
		\limsup_{t \to \infty} \Vert v(t)\Vert_{L^2}^2\le  d^4\lambda_1^2\bigg(\frac{1+\alpha}{c_3}\bigg)^2\Gr^2.
	\end{equation*}
	 
	 Now return to the combined inequality \eqref{eq:w1}, where we apply the Young inequality to $F \Vert w(t)\Vert_{L^2}$. Thus, we have
	\begin{equation*}
		\frac12 \frac{\dd}{\dd t} y(t) + \frac{d_2}2\Vert \nabla v(t)\Vert_{L^2}^2
		+ \frac {c_3}2\Vert w(t)\Vert_{L^2}^2 
		\leq \frac 1{2c_3}F^2.
	\end{equation*}
	Integrating over $[t, t + T]$, and taking the $\limsup$, gives
	\begin{equation*}
		\limsup_{t \to \infty}\bigg(\frac1T \int_{t}^{t + T} \Vert \nabla v(s)\Vert_{L^2}^2 \dd s\bigg) 
		\leq \frac{F^2}{c_3 d} + \limsup_{t \to \infty} \frac1T y(t)
		= \frac{F^2}{c_3 d} + \frac{1}{dT}\bigg(\frac{1+\alpha}{c_3}\bigg)^2 F^2.
	\end{equation*}
	Rearranging the constants, we obtain
	\begin{equation*}
		\limsup_{t \to \infty}\bigg(\frac1T \int_{t}^{t + T} \Vert \nabla v(s)\Vert_{L^2}^2 \dd s\bigg) 
		\leq d^3\lambda_1^2\bigg(\frac{c_3 }{(1+\alpha)^2} + \frac{1}{T}\bigg)\bigg(\frac{1+\alpha}{c_3}\bigg)^2 \Gr^2.
	\end{equation*}
	
	Finally, by the Ladyzhenskaya inequality in two dimensions,
	\begin{equation*}
		\Vert   v(s)\Vert^4_{L^{4}}\leq \Vert   v(s)\Vert^{2}_{L^2} \Vert   \nabla v(s)\Vert^{2}_{L^2}.
	\end{equation*}
	Thus, we have
	\begin{align*}
		\frac{1}{d^2}\bigg(\limsup_{t\to\infty}\frac{1}{T}\int_{t}^{t + T}\Vert   v(s)\Vert^4_{L^4}\dd s\bigg)
		&\leq \frac{1}{d^2}\Big(\limsup_{t\to\infty}\Vert   v(t)\Vert_{L^2}\Big)^2\bigg( \limsup_{t\to\infty}\frac{1}{T}\int_{t}^{t + T}\Vert   \nabla v(s)\Vert^2_{L^2}\dd s\bigg)
		\\
		&\leq  d^5\lambda_1^4 \bigg(\frac{c_3 }{(1+\alpha)^2} + \frac{1}{T}\bigg) \bigg(\frac{1+\alpha}{c_3}\bigg)^4\Gr^4.
	\end{align*}
	
	This completes the proof of the lemma.
\end{proof}
\begin{proof}[Proof of \texorpdfstring{\cref{thm:lambda_main}}]
	We fix $t\ge 0$ and $T>0$.
	The goal is to verify condition \eqref{eq:lambda_etimate1} in \Cref{lem:lambda_etimate1}.
	Combining the estimates \eqref{eq:lambda_intermediate_1}-\eqref{eq:lambda_intermediate_2} of \Cref{lem:lambda_intermediate_1,lem:lambda_intermediate_2}, we obtain
	\begin{align*}
		&\frac{A_2^2}{d^2}
		+\frac{\gamma^2}{d^2}\bigg(\limsup_{t\to\infty}\frac{1}{T}\int_{t}^{t + T}\Big(\Vert   \tu(s)\Vert^4_{L^4} + \Vert   v(s)\Vert^4_{L^4} + \Vert   \tv(s)\Vert^4_{L^4}\Big)\dd s\bigg)
		\\
		&\leq \lambda_1\Gr+ \gamma^2\Big(\frac{d}{2 T}   + \lambda_1 d ^2\Big)\bigg(1 +\frac{d\lambda_1 \sqrt{2}}{c_1c_2}\bigg)^4\Gr^4 + 
		2\gamma^2 d^5\lambda_1^4 \bigg(\frac{c_3 }{(1+\alpha)^2} + \frac{1}{T}\bigg) \bigg(\frac{1+\alpha}{c_3}\bigg)^4\Gr^4.
	\end{align*}

	 It is enough to take 
	 \begin{equation*}
	 	T > \frac{\gamma^2\bigg[ \frac d2 \Big( 1 + \frac{d\lambda_1\sqrt{2}}{c_1c_2}\Big)^4  + 2d^5\lambda_1^4\Big(\frac{1 + \alpha}{c_3}\Big)^4\bigg]\Gr^4}{ \frac1{36}\min\{\lambda_{M}, \lambda_{N}\} -\lambda_1\Gr - \lambda_1 \gamma^2 d^2 \Big(1 + \frac{d\lambda_1 \sqrt{2}}{c_1c_2}\Big)^4\Gr^4 - \frac{2\gamma^2 d^5\lambda_1^4}{c_3^3}(1 + \alpha)^2\Gr^4}.
	 \end{equation*}
	 This is possible provided that
	 \begin{equation*}
	 	\frac1{36}\min\{\lambda_{M}, \lambda_{N}\} > \lambda_1\Gr + \lambda_1 \gamma^2 d^2 \bigg(1 + \frac{d\lambda_1 \sqrt{2}}{c_1c_2}\bigg)^4\Gr^4 + \frac{2\gamma^2 d^5\lambda_1^4}{c_3^3}(1 + \alpha)^2\Gr^4.
	 \end{equation*}
	 For such $T$, by \Cref{lem:lambda_etimate1} the first $(M, N)$ modes are determining for the pair of solutions $(u, v)$ to the system \eqref{eq:gs}. This completes the proof of the main result.
	 
\end{proof}

%
%
\section{Numerical illustration of determining modes}\label{sec:num_exp}

%
%
\begin{figure}[t!]
	\begin{subfigure}[t]{1\textwidth}
		\centering
		\makebox[.1\textwidth][c]{}
		\hspace{0pt}
		\makebox[.1\textwidth][c]{\small  $0$ tu}
		\hspace{20pt}
		\makebox[.1\textwidth][c]{\small  $500$ tu}
		\hspace{20pt}
		\makebox[.1\textwidth][c]{\small  $1500$ tu}
		\hspace{20pt}
		\makebox[.1\textwidth][c]{\small  $2000$ tu}
		\hspace{20pt}
		\makebox[.1\textwidth][c]{\small  $4000$ tu}
		\hspace{20pt}
		\par\smallskip

		\begin{minipage}[t]{0.03\textwidth}
			\rotatebox{90}{\small{Reference}   $u$}
		\end{minipage}
		\begin{minipage}[t]{0.03\textwidth}
			\rotatebox{90}{256 modes}
		\end{minipage}
		\includegraphics[scale = 1]{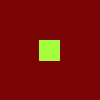}
		\includegraphics[scale = 1]{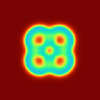}
		\includegraphics[scale = 1]{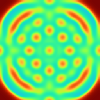}
		\includegraphics[scale = 1]{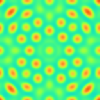}
		\includegraphics[scale = 1]{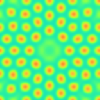}
	\end{subfigure}
	\\[5pt]
	
	\begin{subfigure}[t]{1\textwidth}
		\centering
		\begin{minipage}[t]{0.03\textwidth}
			\rotatebox{90}{\small{Reference}   $v$}
		\end{minipage}
		\begin{minipage}[t]{0.03\textwidth}
			\rotatebox{90}{256 modes}
		\end{minipage}
		\includegraphics[scale = 1]{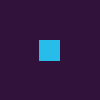}
		\includegraphics[scale = 1]{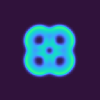}
		\includegraphics[scale = 1]{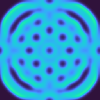}
		\includegraphics[scale = 1]{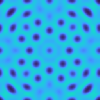}
		\includegraphics[scale = 1]{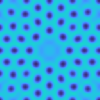}
	\end{subfigure}
	\hfill	
	\\[5pt]
	\begin{subfigure}[c]{1.\textwidth}
		\hspace{161pt}
		\includegraphics[scale = 1]{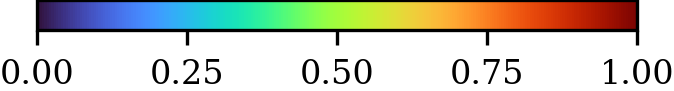}
	\end{subfigure}
	\hfill
	
	\caption{Snapshots of the reference solution: Gray--Scott model with $F = 0.039$, $k = 0.058$ (Holes pattern). Colormap shows concentration $u$ or $v$.}
	\label{fig:holes}
\end{figure}

%
%
\begin{figure}
	\begin{subfigure}[t]{1\textwidth}
		
		\centering
		\makebox[.1\textwidth][c]{}
		\hspace{0pt}
		\makebox[.1\textwidth][c]{\small  $0$ tu}
		\hspace{20pt}
		\makebox[.1\textwidth][c]{\small  $12.5$ tu}
		\hspace{20pt}
		\makebox[.1\textwidth][c]{\small  $25$ tu}
		\hspace{20pt}
		\makebox[.1\textwidth][c]{\small  $50$ tu}
		\hspace{20pt}
		\makebox[.1\textwidth][c]{\small  $100$ tu}
		\hspace{20pt}
		\par\smallskip

		\begin{minipage}[t]{0.03\textwidth}
			\rotatebox{90}{\small{Reference}   $u$}
		\end{minipage}
		\begin{minipage}[t]{0.03\textwidth}
			\rotatebox{90}{256 modes}
		\end{minipage}
		\includegraphics[scale = 1]{u0000}
		\includegraphics[scale = 1]{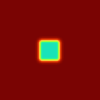}
		\includegraphics[scale = 1]{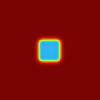}
		\includegraphics[scale = 1]{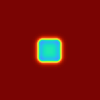}
		\includegraphics[scale = 1]{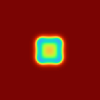}
	\end{subfigure}
	\\[5pt]
	
	\begin{subfigure}[t]{1\textwidth}
		\centering
		\begin{minipage}[t]{0.03\textwidth}
			\rotatebox{90}{\small{Assimilated}   $\tu$}
		\end{minipage}
		\begin{minipage}[t]{0.03\textwidth}
			\rotatebox{90}{256 modes}
		\end{minipage}
		\includegraphics[scale = 1]{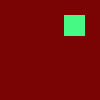}
		\includegraphics[scale = 1]{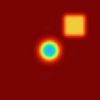}
		\includegraphics[scale = 1]{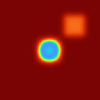}
		\includegraphics[scale = 1]{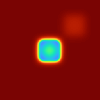}
		\includegraphics[scale = 1]{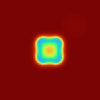}
	\end{subfigure}
	\hfill	
	\\[5pt]
	\begin{subfigure}[c]{1.\textwidth}
		\hspace{161pt}
		\includegraphics[scale = 1]{colorbar_turbo}
	\end{subfigure}
	\hfill
	
	\caption{Snapshots of the reference solution (first row) and the assimilated solution (second row): Gray--Scott model with $F = 0.039$, $k = 0.058$ (Holes pattern).  The number of observed modes is $M = 12$. Colormap shows concentration $u$ or $\tu$.}
	\label{fig:holes_pert}
\end{figure}

\subsection{The reference system}

To illustrate the determining mode condition in Theorem~\ref{thm:lambda_main}, we consider a particular case of the system \eqref{eq:gs}, also called the Gray--Scott model, on the unit square $D\coloneqq [0,1]\times[0,1]$ that is given by
\begin{equation}\label{eq:gs_ref}
	\left\{
	\begin{aligned}
		\partial_t{u} &= d_1\Delta u  -uv^2 + F(1-u),
		\\
		\partial_t{v} &= d_2\Delta v  + uv^2 - (F+k)v,
	\end{aligned}
	\right.
\end{equation}
with parameters $d_1= 1.6\times 10^{-5}$, $d_2 = 0.8\times 10^{-5}$ and homogeneous Neumann boundary conditions for both species.

Since the system~\eqref{eq:gs_ref} has no closed-form solution, we compute the dynamics numerically using a semi-implicit scheme combined with a pseudospectral Fourier method. 
The simulation is performed on a 256x256 cell-centered grid over $D$ and a fixed time step $\Delta t = 0.5$.

We fix the parameters to $F = 0.039$ and $k = 0.058$, and initialize the system at the  steady state $u \equiv 1$, $v \equiv 0$. 
To trigger pattern formation, we introduce a localized perturbation by setting $u \equiv 0.50$ and $v \equiv 0.25$ in a central square of side length $0.2$. As seen in \Cref{fig:holes}, the system evolves toward a characteristic pattern known in the literature as \textit{holes}.

The numerical solution of \eqref{eq:gs_ref} is taken as the \textit{reference solution}, whose dynamics will be compared with a second system that we refer to as an assimilated system.

\subsection{The assimilated system}

In general, the theory of determining modes does not specify how to construct the forces $g_1,g_2$ and $\tg_1, \tg_2$.  Because the conditions
\[
\limsup_{t\to\infty}  \Big(\Vert (g_1-\tg_1)(t)\Vert_{L^2}+ \Vert (g_2-\tg_2)(t)\Vert_{L^2}\Big) = 0 \mbox{ and } \lim_{t\to\infty}\Big(\Vert P_{M}(u - \tu)(t)\Vert_{L^2} +  \Vert P_{N}(v - \tv)(t)\Vert_{L^2}\Big)= 0
\]
must hold, these forces must be defined appropriately.

Typically they are defined as  nonlinear feedbacks like in the work of Olson and Titi~\cite{olson2003determining} or linear feedbacks like in the work of Azouani, Olson, and Titi~\cite{azouani2014continuous}.

To illustrate the determining mode condition in \Cref{thm:lambda_main}, we define a continuous data assimilation algorithm with linear feedback adapted to the problem~\eqref{eq:gs_ref} using observations consisting of the first $M$ modes.

Let $(u, v)$ be the numerical solution of \eqref{eq:gs_ref}. 
We introduce an \textit{assimilated system} $(\tu, \tv)$, defined by
\begin{equation}\label{eq:nudgedgs}
\left\{
\begin{aligned}
\partial_t{\tu} &= d_1\Delta \tu  -\tu\tv^2 + F(1-\tu) + \sigma u\xi_1 + 2 P_M(   u -  \tu),
\\
\partial_t{\tv} &= d_2\Delta \tv  + \tu\tv^2 - (F+k)\tv + \sigma v\xi_2 + 2 P_M(  v -   \tv).
\end{aligned}
\right.
\end{equation}
where $P_M$ denotes the projection onto the first $M$ modes.  $\xi_1 = \xi_1(t,x)$ and $\xi_2 = \xi_2(t,x)$ are filtered Gaussian random field supported in the finite time interval $[0, 2000]$. The filtering is given in Fourier space by $\hat\xi_j(k) = (5 + \vert k\vert^2)^{-\beta}Z_k$, where $Z_k$ are independent complex Gaussian variables, and $\beta>0$ controls the spatial regularity of the noise. Larger values of $\beta$ produce smoother noise fields.

The term $P_M(   u -  \tu)$ (and similarly for $v$) is a feedback that acts only on the first $M$ modes and serves to enforce agreement on the observed low modes.

We perform four experiments to demonstrate numerically the result in \Cref{thm:lambda_main}:
\begin{itemize}
\item Experiment 1: $\sigma = 0$ (no noise)
\item Experiment 2: $\sigma = 0.25$, $\beta = 0.5$ (rough noise)
\item Experiment 3: $\sigma = 0.25$, $\beta = 1.0$ (moderate noise)
\item Experiment 4: $\sigma = 0.25$, $\beta = 1.5$ (smooth noise)
\end{itemize}

The assimilated system~\eqref{eq:nudgedgs} starts with different initial conditions from those of the reference solution. Specifically, we take $\tilde u \equiv 1$ and $\tilde v \equiv 0$ everywhere, and introduce a localized perturbation by setting $\tilde u \equiv 0.40$ and $\tilde v \equiv 0.15$ in a square region of side length $0.2$, centered at $(0.25,0.25)$, see \Cref{fig:holes_pert}.
The snapshots in \Cref{fig:holes_pert} also illustrates the synchronization process, showing that the assimilated solution rapidly recovers the spatial structure of the reference solution and becomes visually indistinguishable after 100 time units (tu).

\begin{figure}[t!]
	\centering
	{
		\subcaption*{\bf Errors between $P_Mu$ and $P_M\tu$}
		\includegraphics[scale=0.4]{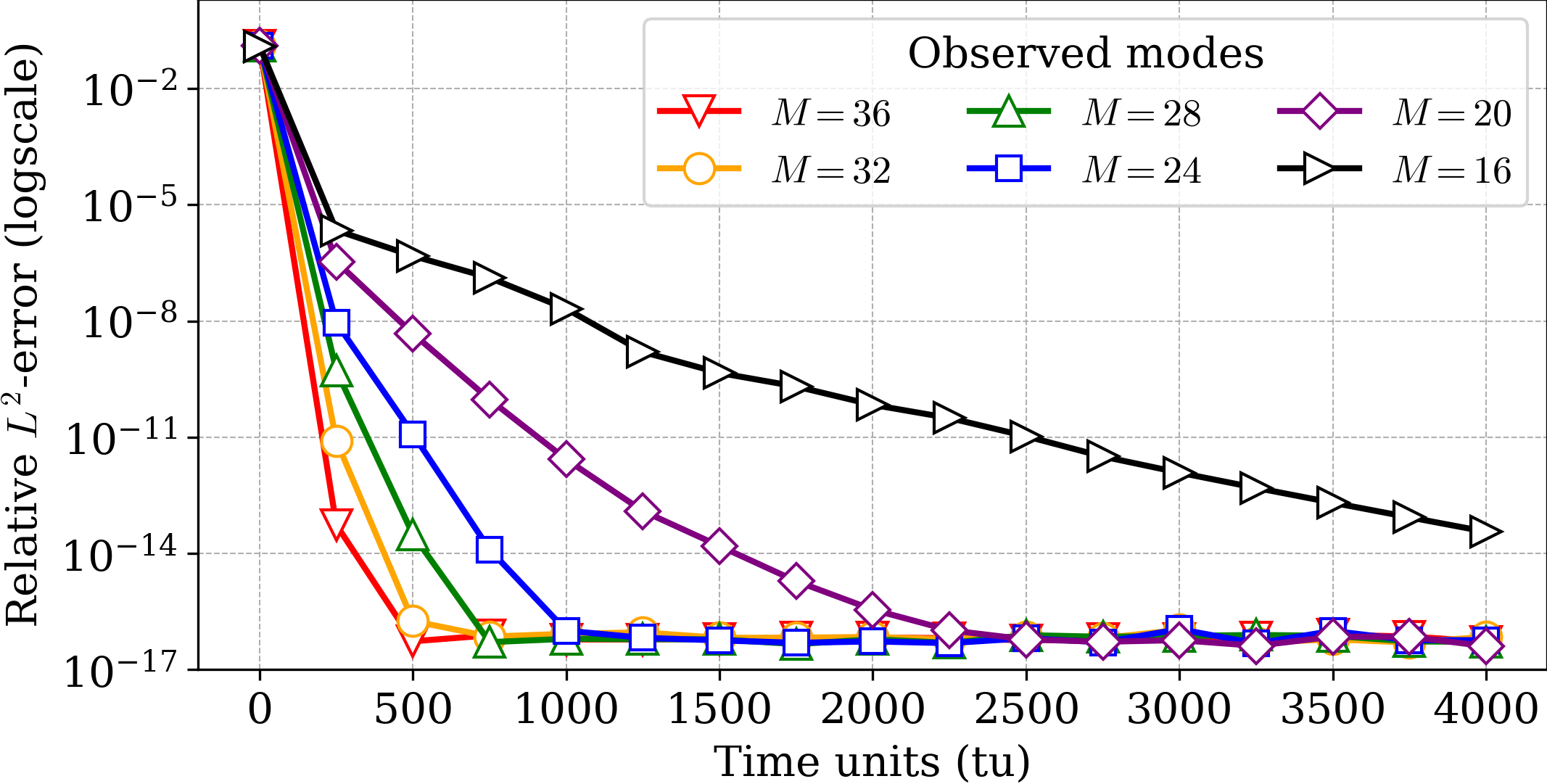}
		\;
		\includegraphics[scale=0.4]{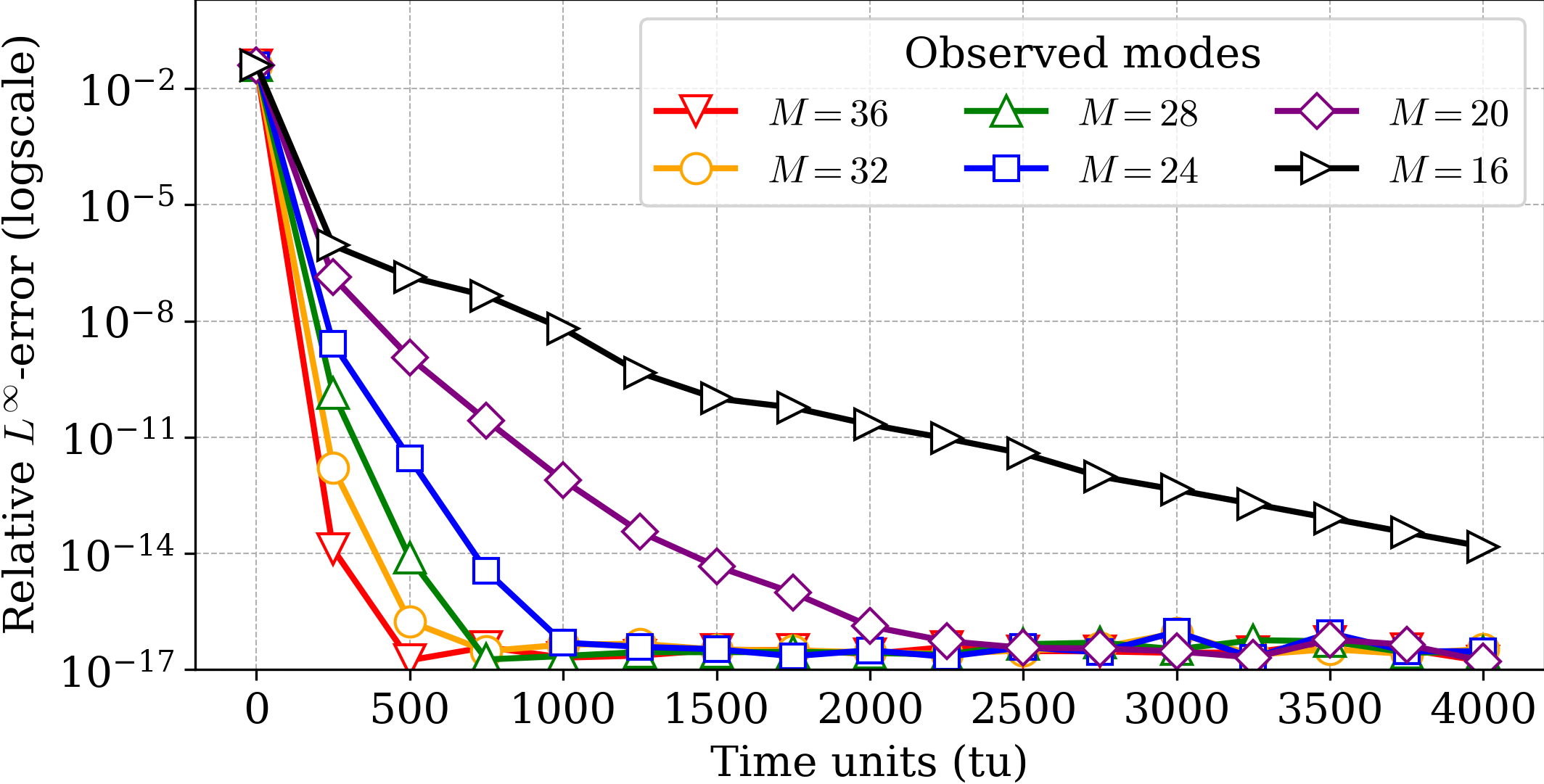}
		
		\subcaption*{\bf Errors between $u$ and $\tu$}
		\includegraphics[scale=0.4]{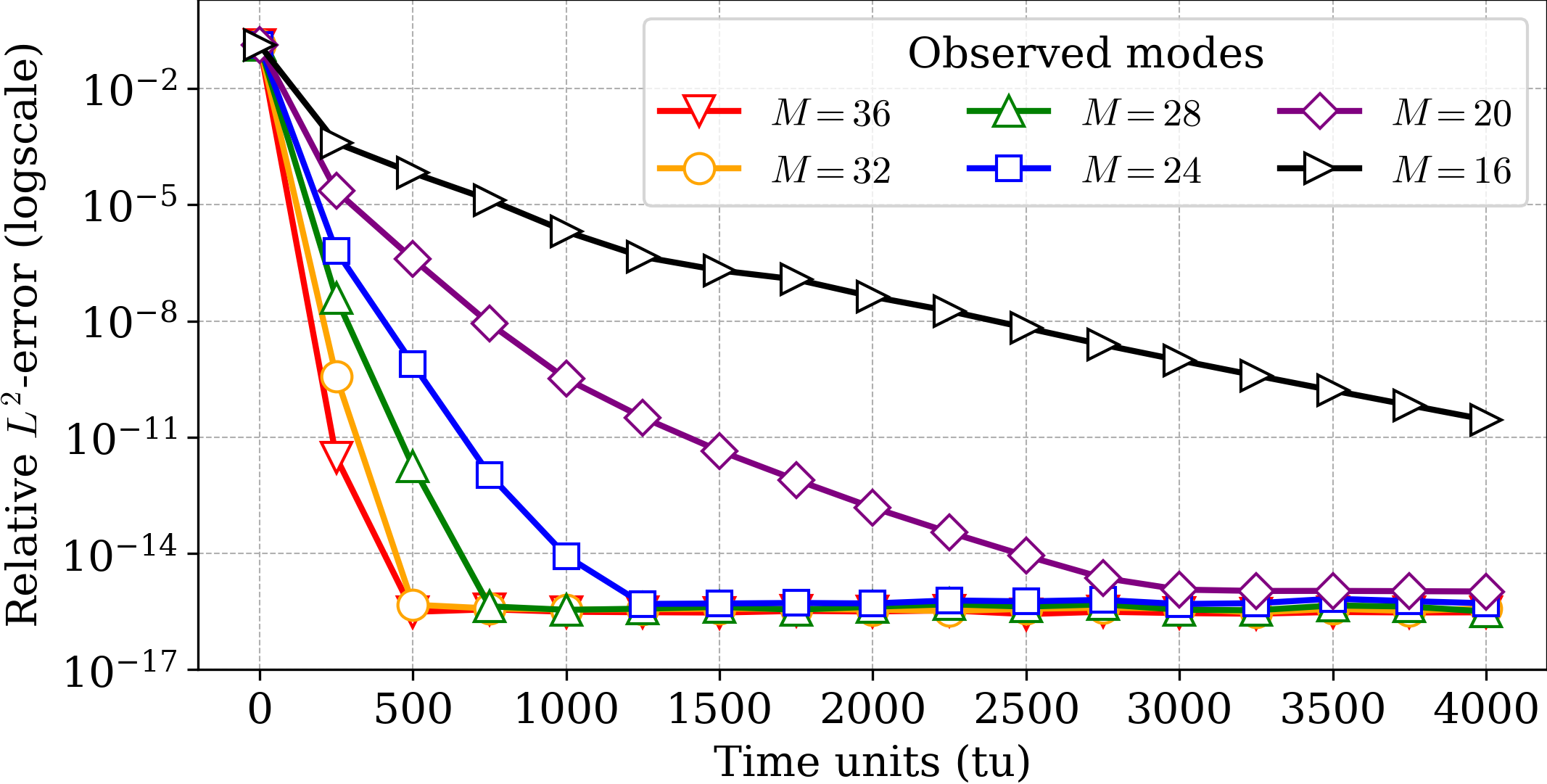}
		\;
		\includegraphics[scale=0.4]{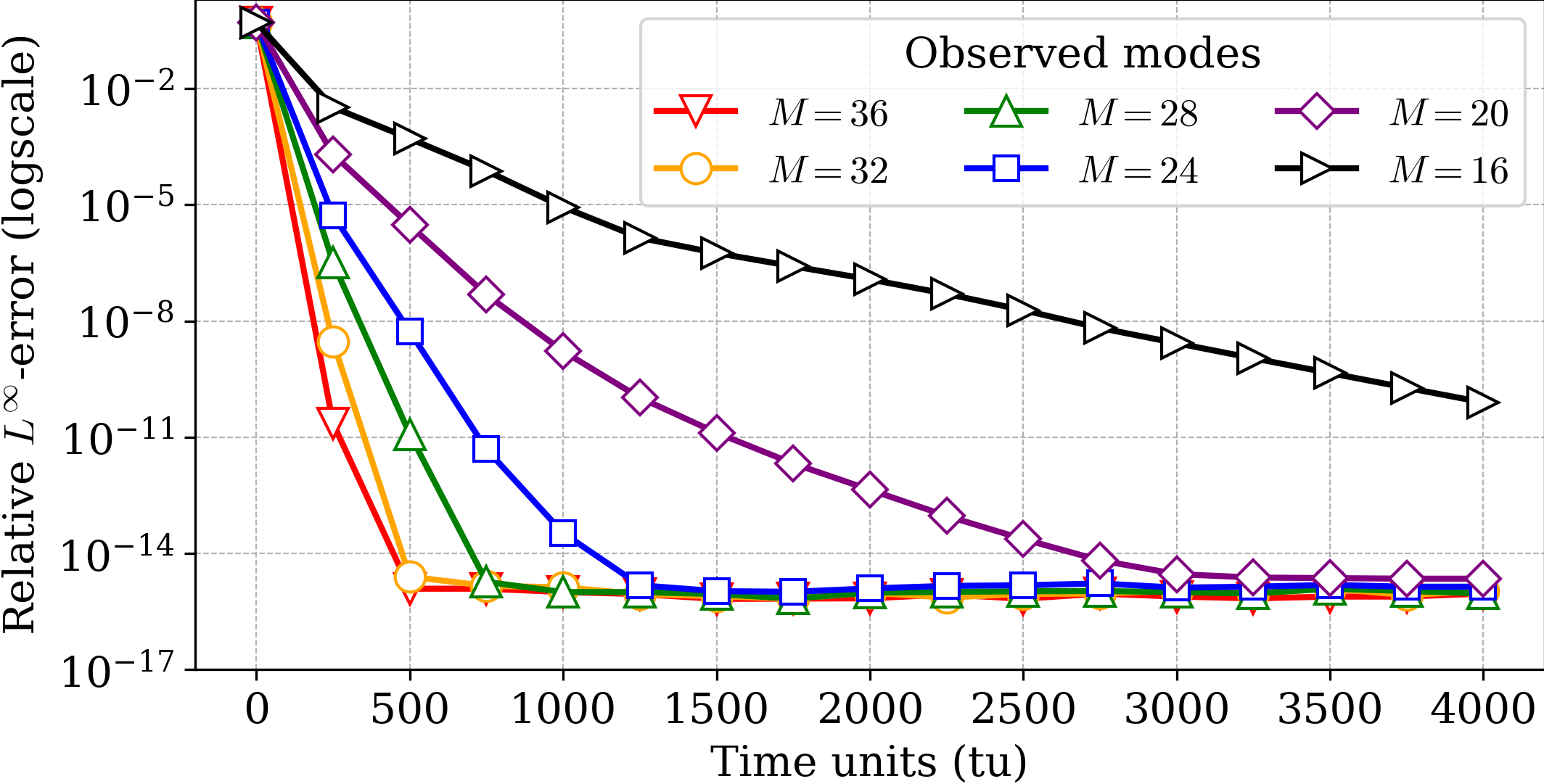}
	}
	
	\caption{\textbf{Experiment 1.} Evolution of the relative $L^2$-error (left) and the $L^\infty$-error (right) between the reference solution and the assimilated solution.}
	\label{fig:experiment1_errors}
\end{figure}

We monitor the convergence between the reference and assimilated solutions by tracking the relative error in both the $L^2$-norm ($\Vert u(t) - \tu(t)\Vert_{L^2}/\Vert u(t)\Vert_{L^2}$), and $L^\infty$-norms ($\Vert u(t) - \tu(t)\Vert_{L^\infty}/\Vert u(t)\Vert_{L^\infty}$) over time.

\Cref{fig:experiment1_errors,fig:experiment2_errors,fig:experiment3_errors,fig:experiment4_errors} illustrate the convergence behavior of the assimilated system under deterministic and stochastic forcing. In the deterministic case, see \Cref{fig:experiment1_errors}, the error decays exponentially to zero once a sufficient number of modes is used, and synchronization, i.e., convergence to machine precision ($\approx 10^{-16}$), is achieved for $M \ge 24$. This confirms that only a small number of low Fourier modes is sufficient to determine the full dynamics.

In the presence of stochastic forcing, see \Cref{fig:experiment2_errors,fig:experiment3_errors,fig:experiment4_errors}, the convergence behavior changes due to the persistent influence of noise. In particular, an initial transient phase is observed during which the error remains nearly constant. After this phase, the error begins to decay, and synchronization is eventually achieved provided that a sufficiently large number of modes is used. For rough noise, see \Cref{fig:experiment2_errors}, this requires approximately $M\ge 32$, indicating that more modes are needed to compensate for the increased excitation of higher frequencies.

As the noise becomes smoother, see \Cref{fig:experiment3_errors,fig:experiment4_errors}, the convergence behavior improves significantly. The transient phase shortens, and synchronization is achieved for smaller values of $M$, with a threshold returning to $M\ge 24$ for sufficiently regular noise. This shows that the number of determining modes depends on the spectral properties of the noise, including both its regularity and intensity. Rougher and stronger noise excites higher modes and requires more observed modes, whereas smoother and weaker noise mainly affects large scales, which are more effectively controlled by the feedback mechanism.

\begin{figure}[t!]
	\centering
	{
		\subcaption*{\bf Errors between $P_Mu$ and $P_M\tu$}
		\includegraphics[scale=0.4]{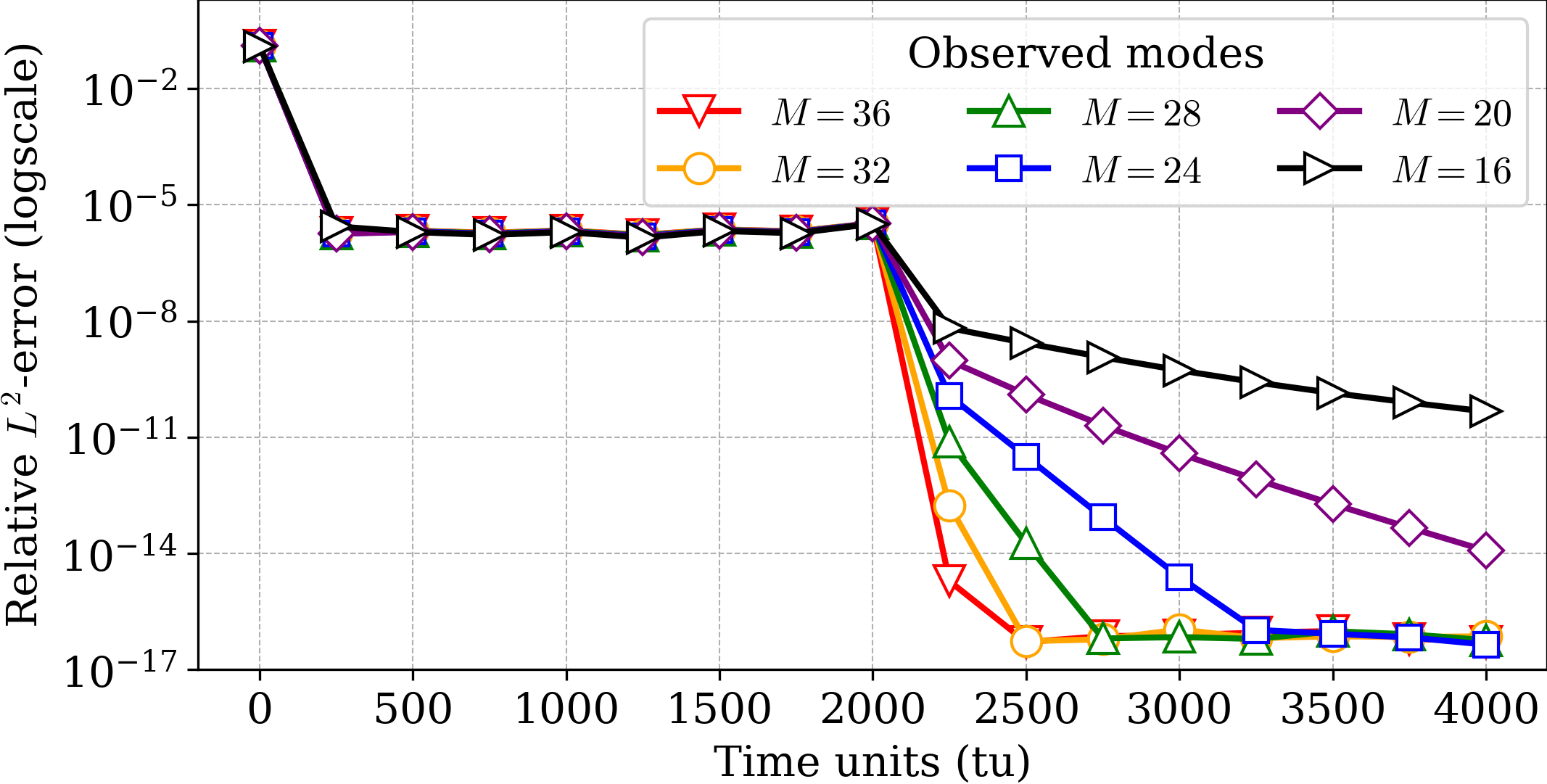}
		\;
		\includegraphics[scale=0.4]{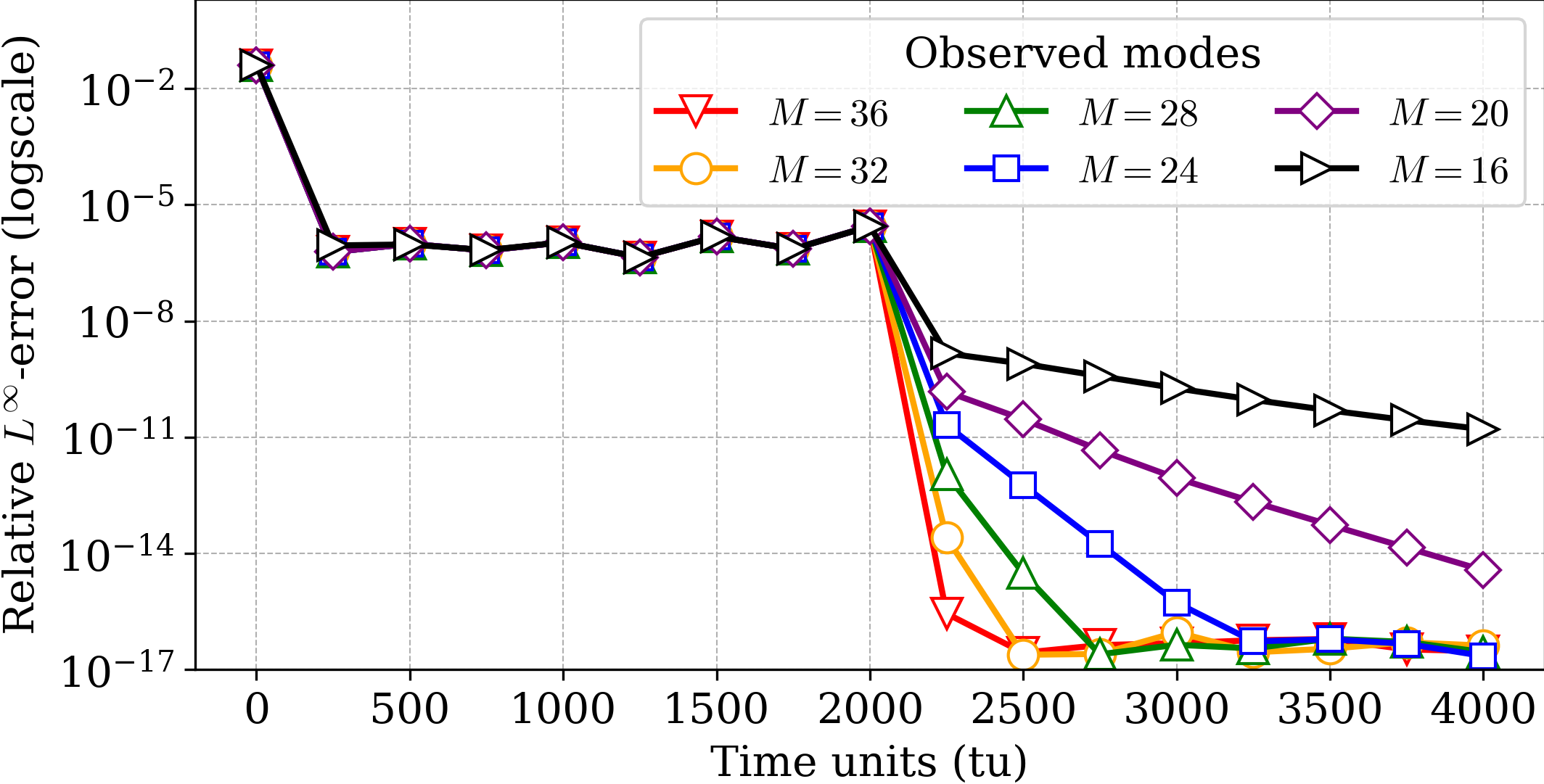}
		
		\subcaption*{\bf Errors between $u$ and $\tu$}
		\includegraphics[scale=0.4]{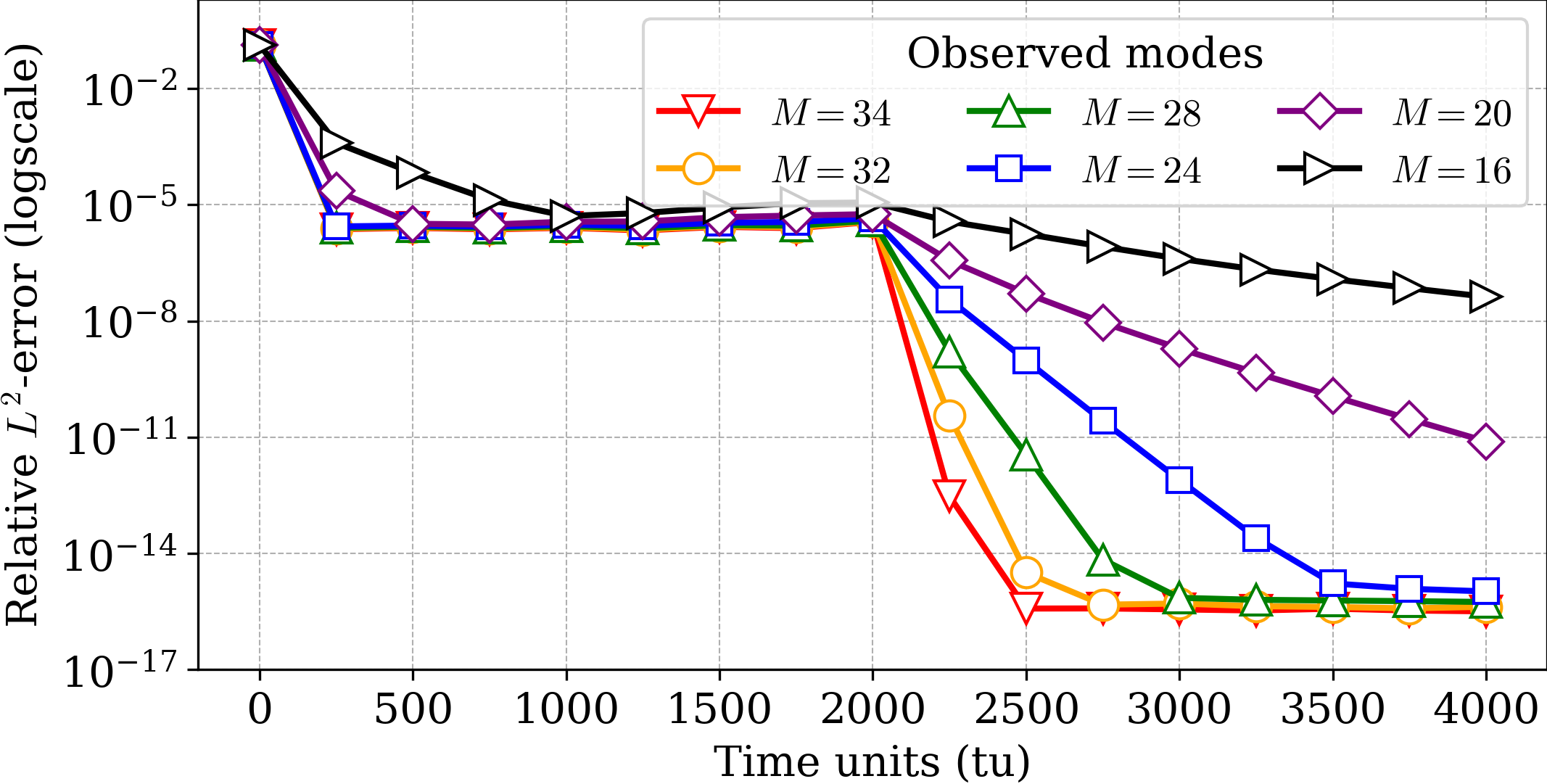}
		\;
		\includegraphics[scale=0.4]{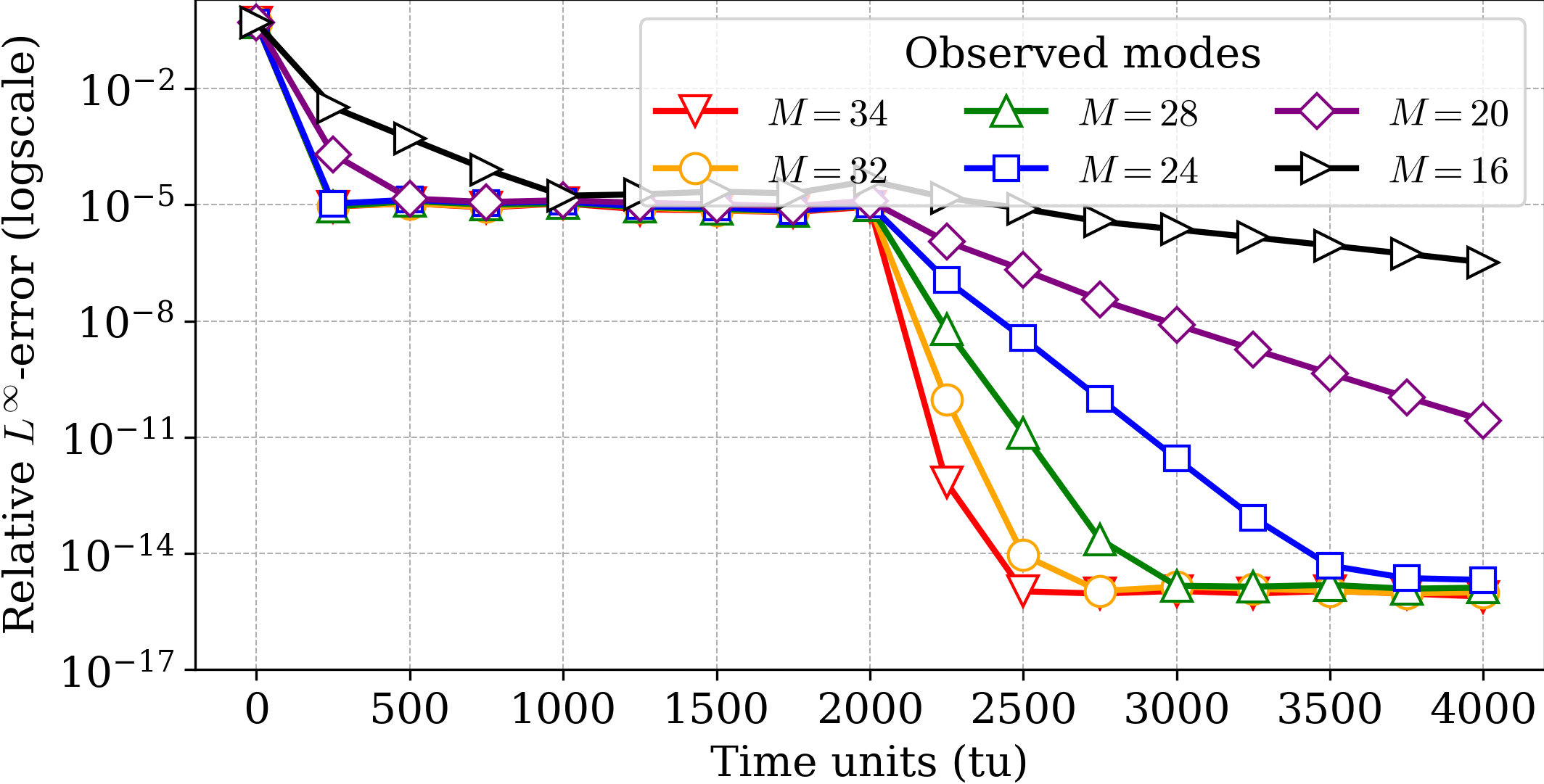}
	}
	
	\begin{subfigure}[t]{1\textwidth}
		
		\centering
		\makebox[.1\textwidth][c]{}
		\hspace{0pt}
		\makebox[.1\textwidth][c]{\small  $500$ tu}
		\hspace{20pt}
		\makebox[.1\textwidth][c]{\small  $1000$ tu}
		\hspace{20pt}
		\makebox[.1\textwidth][c]{\small  $1500$ tu}
		\hspace{20pt}
		\makebox[.1\textwidth][c]{\small  $2000$ tu}
		\hspace{10pt}
		\makebox[.1\textwidth][c]{}
		\hspace{20pt}
		\par\smallskip

		\begin{minipage}[t]{0.03\textwidth}
			\rotatebox{90}{\small{Noise roughness}}
		\end{minipage}
		\begin{minipage}[t]{0.03\textwidth}
			\rotatebox{90}{ $\beta = 0.5$}
		\end{minipage}
		\includegraphics[scale = 1]{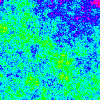}
		\includegraphics[scale = 1]{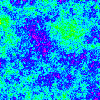}
		\includegraphics[scale = 1]{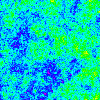}
		\includegraphics[scale = 1]{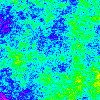}
		\includegraphics[scale = 1]{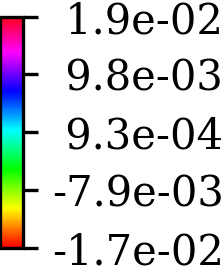}
	\end{subfigure}
	
	\caption{\textbf{Experiment 2.} Evolution of the relative $L^2$-error (left) and the $L^\infty$-error (right) between the reference solution and the assimilated solution. The last row shows snapshots of the filtered Gaussian random field used.}
	\label{fig:experiment2_errors}
\end{figure}

\subsection{Interpretation}

The numerical experiments are consistent with the conclusion of \Cref{thm:lambda_main}, which states that, under some conditions, a finite number of low modes determines the full long-time dynamics of the system. In particular, the numerical experiments suggest that, for the model considered, the first 24-32 modes, depending on the characteristic of the perturbation or feedback, are determining, in the sense of \Cref{defn:determining_modes}, for the dynamics. 

However, the experiments should not be interpreted as a sharp quantitative verification of the bound in \Cref{thm:lambda_main}. \Cref{thm:lambda_main} announces that if 
\begin{equation*}
	\frac{\min\{\lambda_{M},\lambda_{N}\}}{36\lambda_1}
	>  \Gr +   \gamma^2 \bigg[d^2\Big(1 + \frac{d\lambda_1 \sqrt{2}}{c_1c_2}\Big)^4  + \frac{2d^5\lambda_1^3}{c_3^3}(1 + \alpha)^2\bigg]\Gr^4.
\end{equation*}
then the first $M$ modes are determining.
On $D = [0,1]\times[0,1]$, using $\lambda_k\sim 4\pi k$. For the model considered in this experiment, $\Gr\approx 1.5\times 10^{8}$, and the dominant contribution comes from the term $\Gr^4$. This leads to an estimate $M\approx 10^{22}$. The large value comes mainly from $1/d^2$.

The analytical estimate is very conservative and provides only a sufficient condition. In practice, synchronization is observed for much smaller values of $M$, namely $M\approx24$ in the deterministic case and $M\approx 32$ in the rough-noise case. This deviation is typical in determining modes theory, see~\cite{olson2003determining}: rigorous bounds guarantee convergence but often overestimate the number of modes required in computations.

\begin{figure}[t!]
	\centering
	{
		\subcaption*{\bf Errors between $P_Mu$ and $P_M\tu$}
		\includegraphics[scale=0.4]{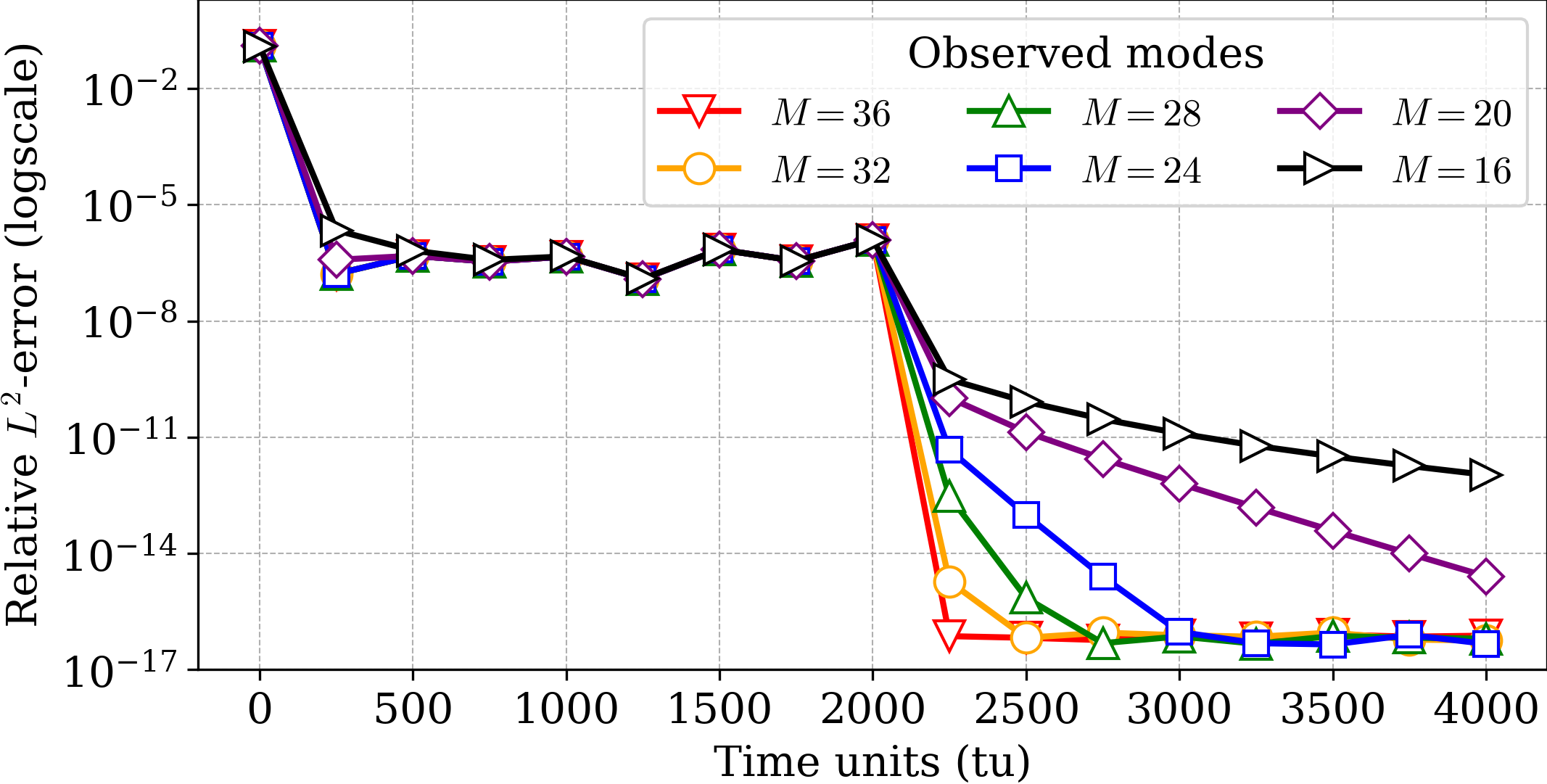}
		\;
		\includegraphics[scale=0.4]{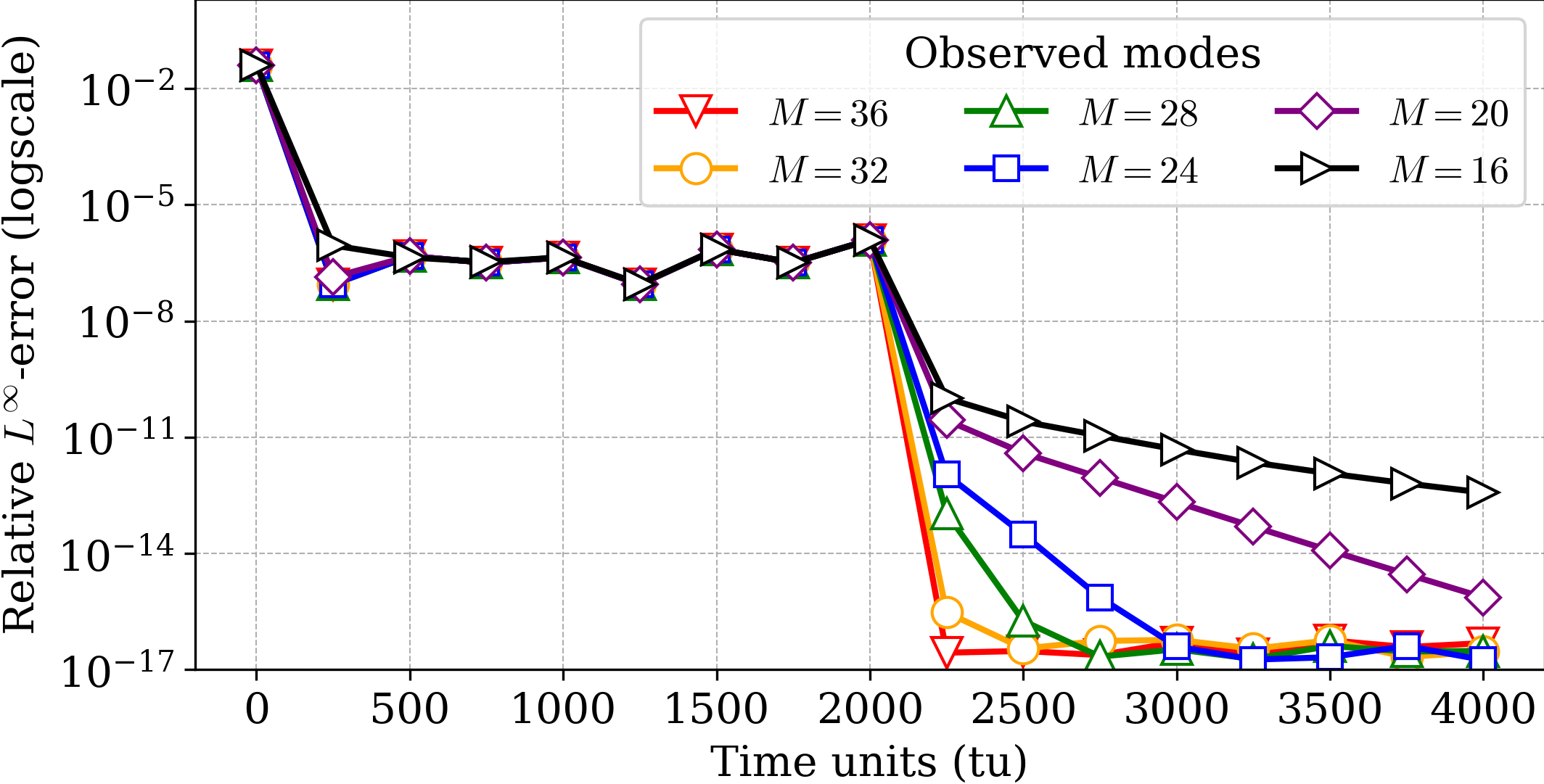}
		
		\subcaption*{\bf Errors between $u$ and $\tu$}
		\includegraphics[scale=0.4]{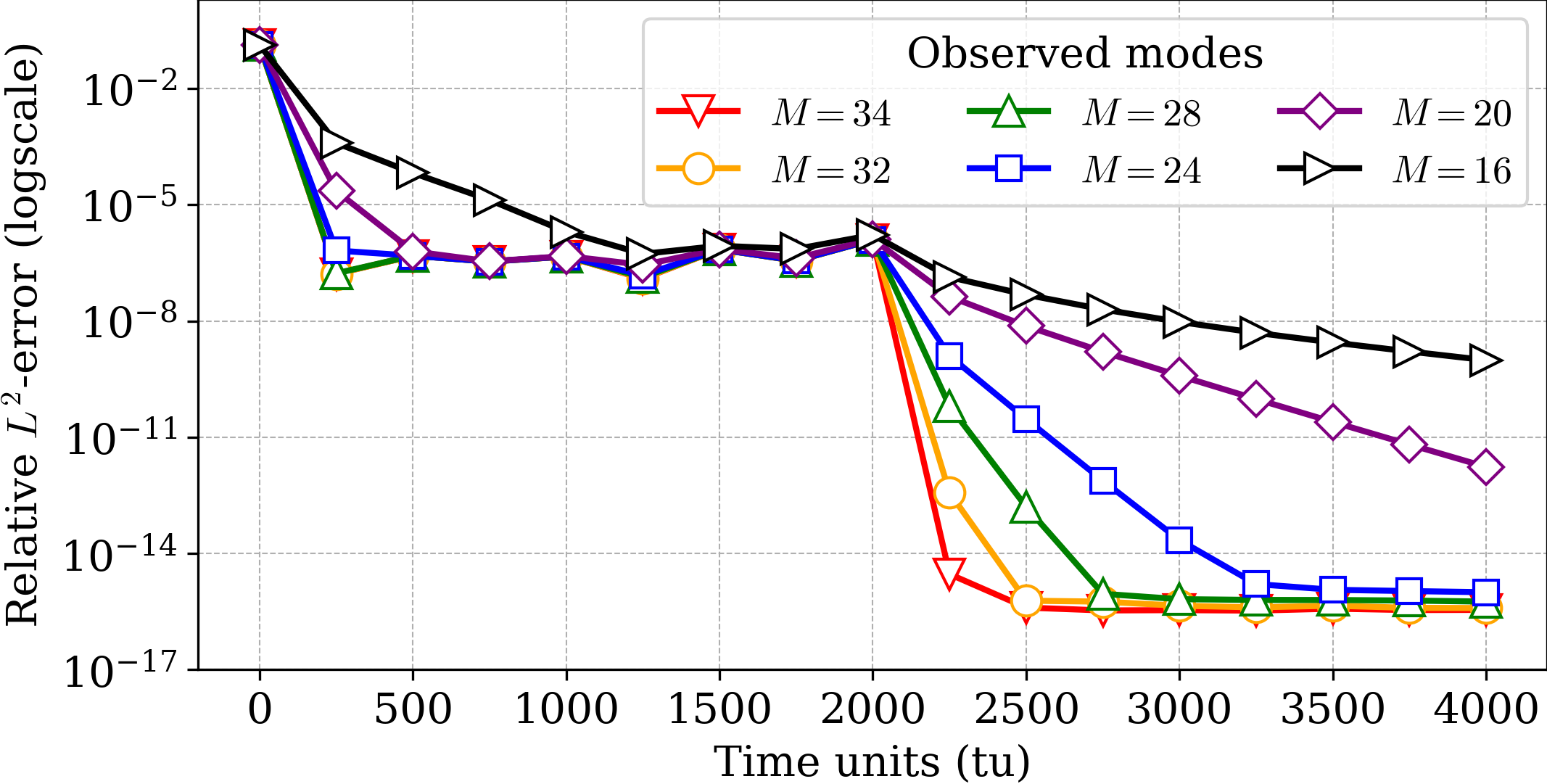}
		\;
		\includegraphics[scale=0.4]{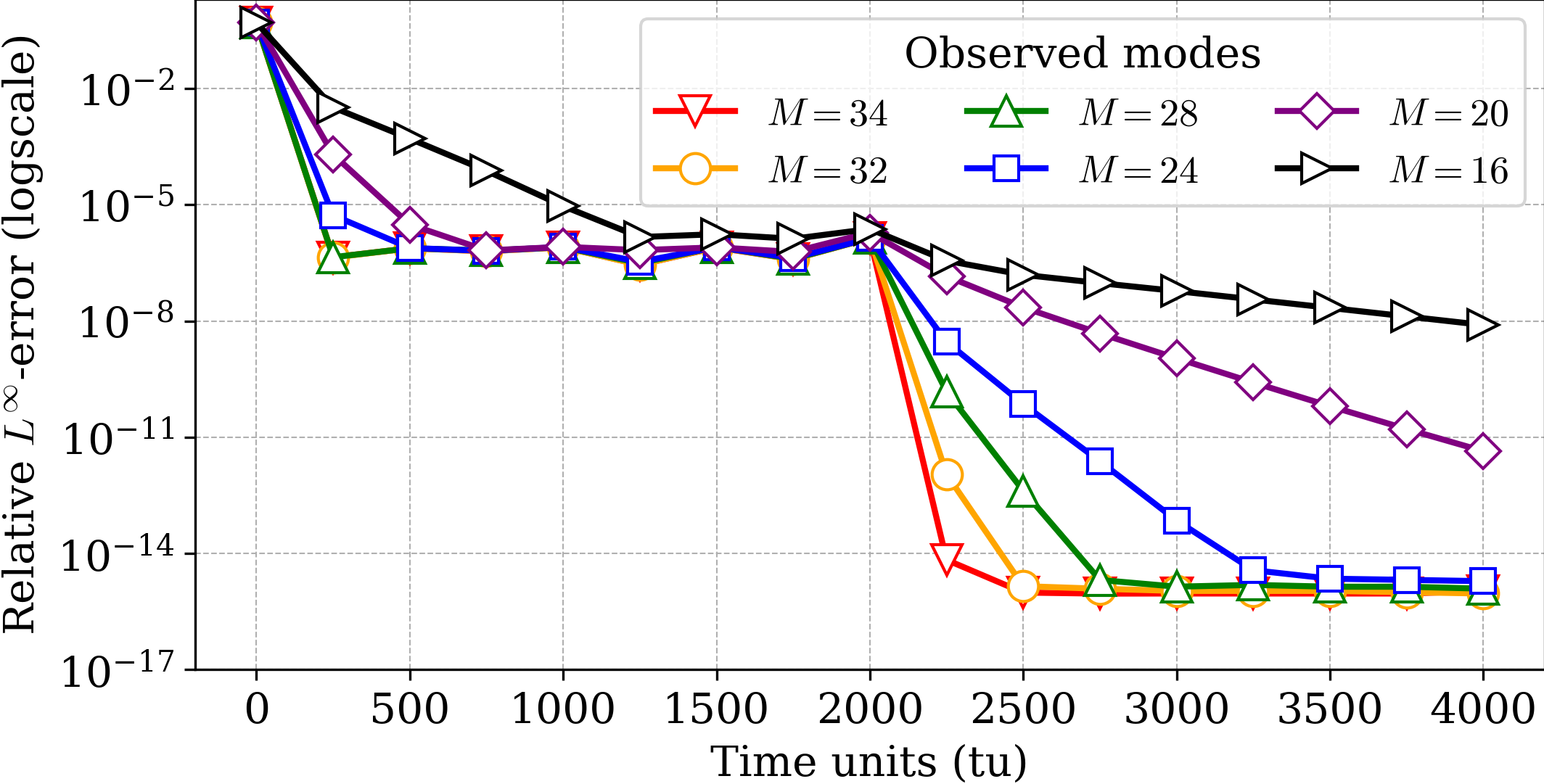}
	}
	
	\begin{subfigure}[t]{1\textwidth}
		
		\centering
		\makebox[.1\textwidth][c]{}
		\hspace{0pt}
		\makebox[.1\textwidth][c]{\small  $500$ tu}
		\hspace{20pt}
		\makebox[.1\textwidth][c]{\small  $1000$ tu}
		\hspace{20pt}
		\makebox[.1\textwidth][c]{\small  $1500$ tu}
		\hspace{20pt}
		\makebox[.1\textwidth][c]{\small  $2000$ tu}
		\hspace{10pt}
		\makebox[.1\textwidth][c]{}
		\hspace{20pt}
		\par\smallskip

		\begin{minipage}[t]{0.03\textwidth}
			\rotatebox{90}{\small{Noise roughness}}
		\end{minipage}
		\begin{minipage}[t]{0.03\textwidth}
			\rotatebox{90}{ $\beta = 1.0$}
		\end{minipage}
		\includegraphics[scale = 1]{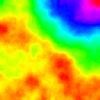}
		\includegraphics[scale = 1]{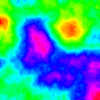}
		\includegraphics[scale = 1]{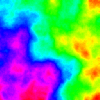}
		\includegraphics[scale = 1]{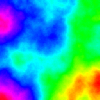}
		\includegraphics[scale = 1]{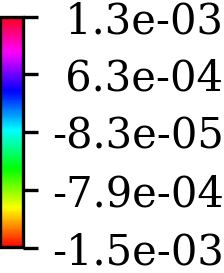}
	\end{subfigure}
	
	\caption{\textbf{Experiment 3.} Evolution of the relative $L^2$-error (left) and the $L^\infty$-error (right) between the reference solution and the perturbed solution. The last row shows snapshots of the filtered Gaussian random field used.}
	\label{fig:experiment3_errors}
\end{figure}
\begin{figure}[t!]
	\centering
	{
		\subcaption*{\bf Errors between $P_Mu$ and $P_M\tu$}
		\includegraphics[scale=0.4]{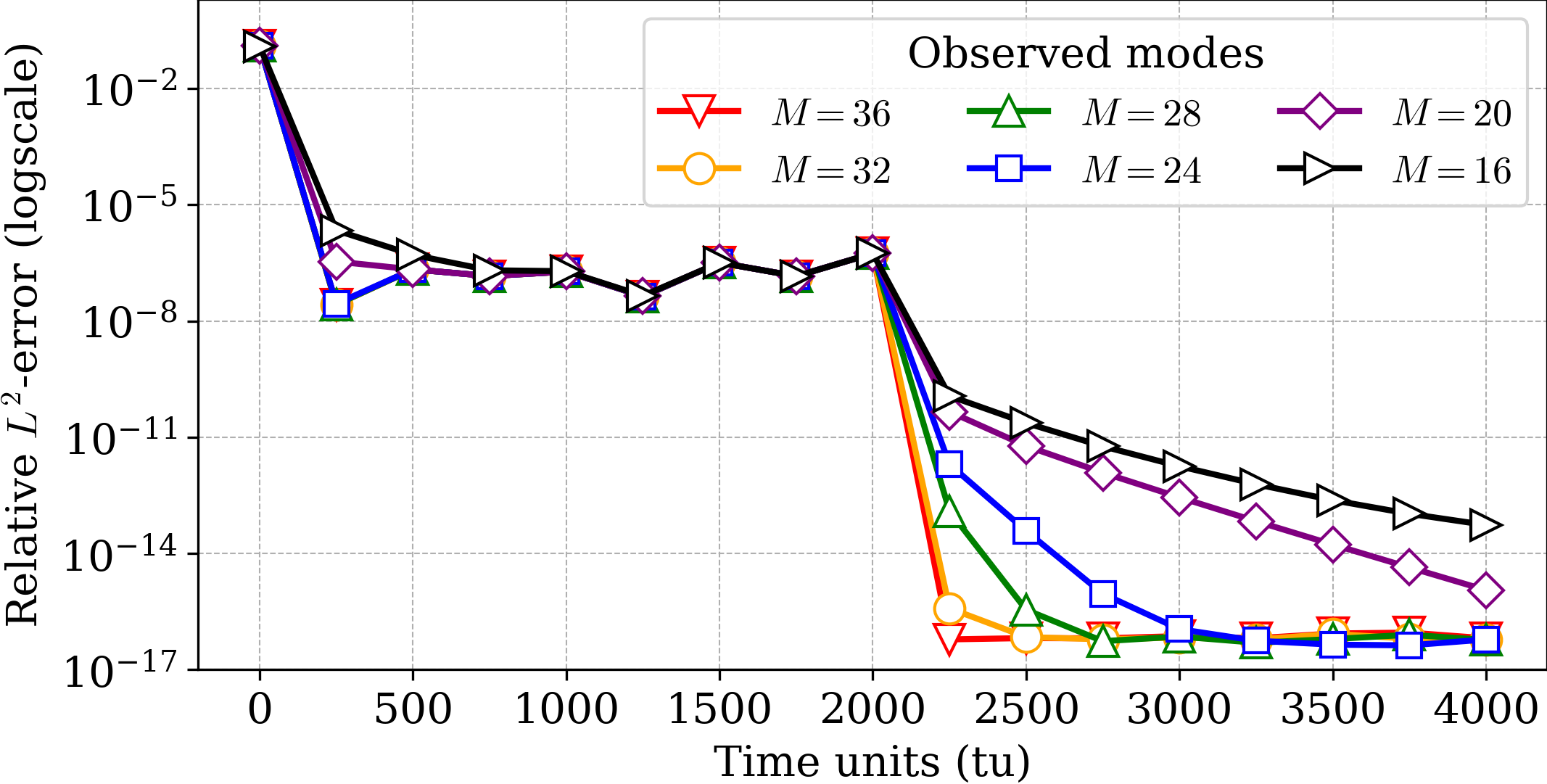}
		\;
		\includegraphics[scale=0.4]{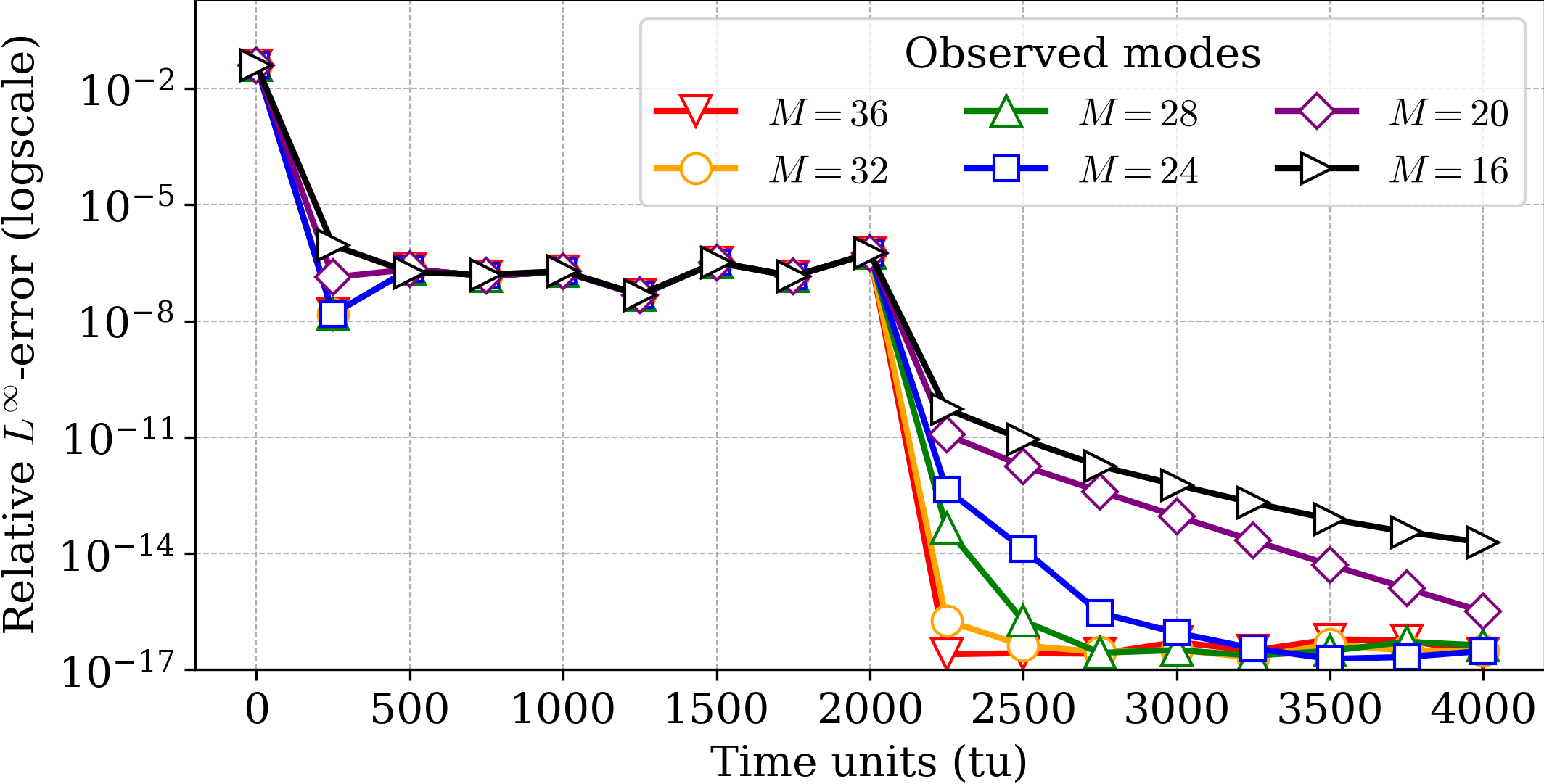}
		
		\subcaption*{\bf Errors between $u$ and $\tu$}
		\includegraphics[scale=0.4]{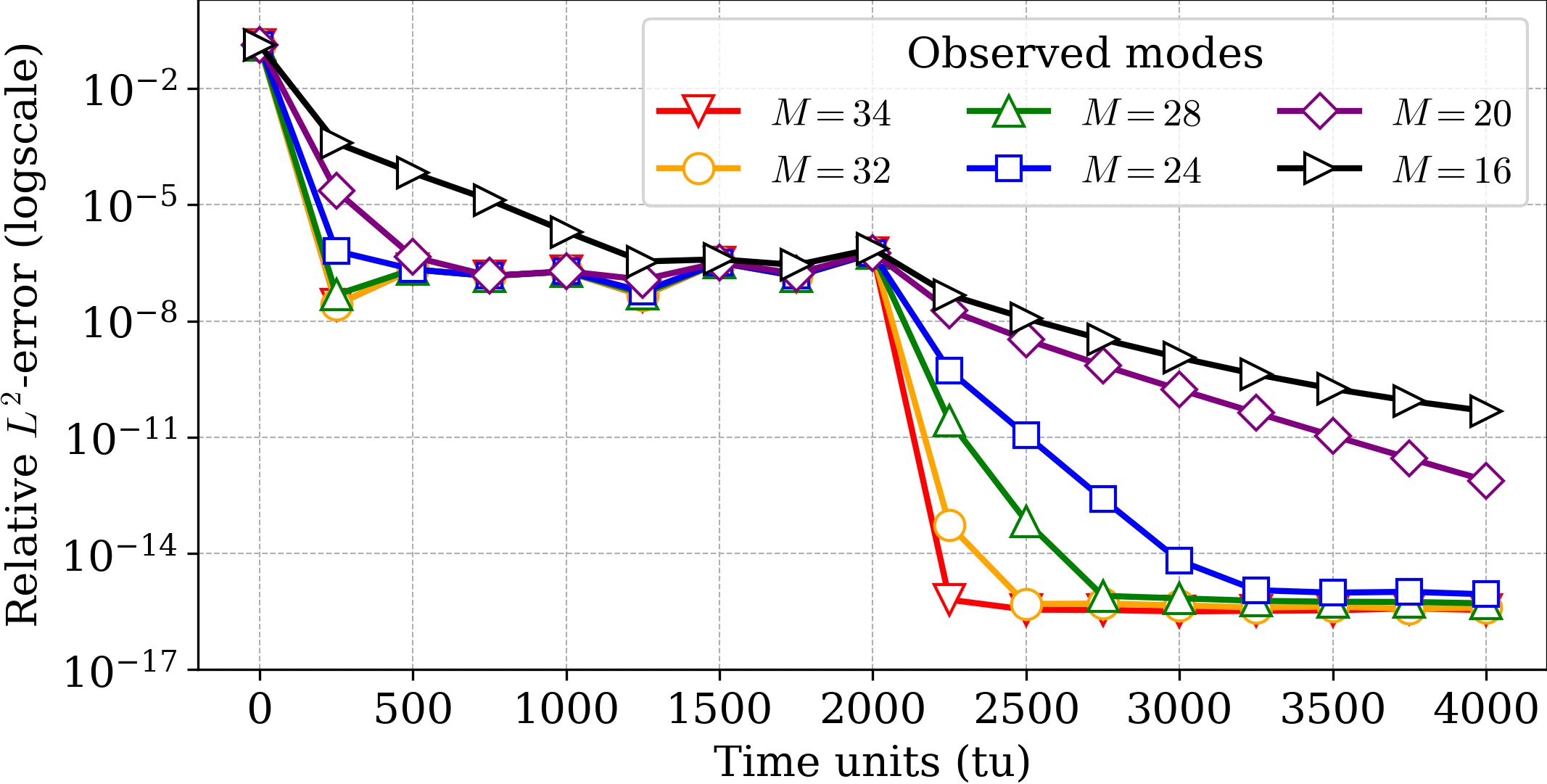}
		\;
		\includegraphics[scale=0.4]{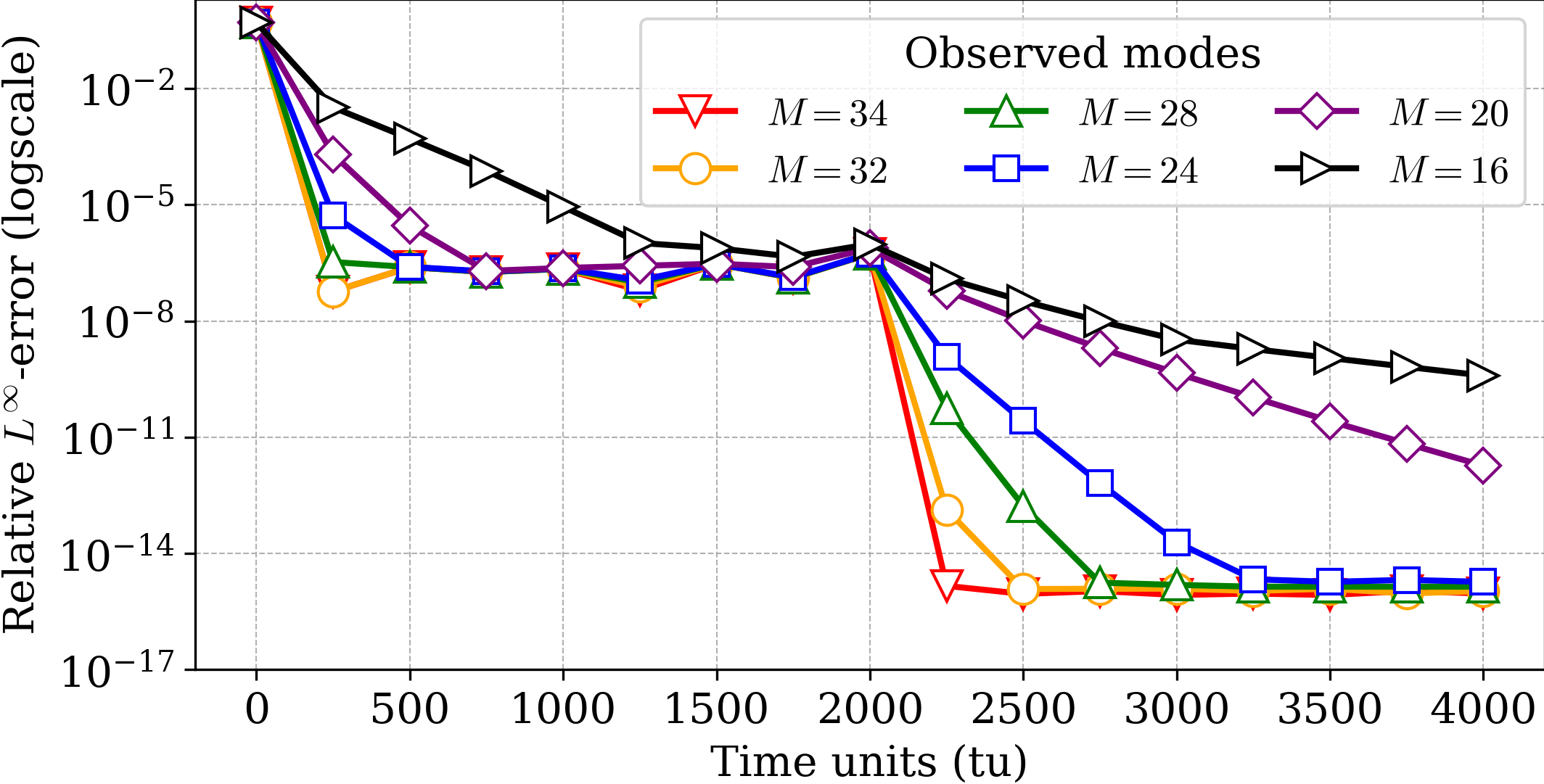}
	}
	
	\begin{subfigure}[t]{1\textwidth}
		
		\centering
		\makebox[.1\textwidth][c]{}
		\hspace{0pt}
		\makebox[.1\textwidth][c]{\small  $500$ tu}
		\hspace{20pt}
		\makebox[.1\textwidth][c]{\small  $1000$ tu}
		\hspace{20pt}
		\makebox[.1\textwidth][c]{\small  $1500$ tu}
		\hspace{20pt}
		\makebox[.1\textwidth][c]{\small  $2000$ tu}
		\hspace{10pt}
		\makebox[.1\textwidth][c]{}
		\hspace{20pt}
		\par\smallskip

		\begin{minipage}[t]{0.03\textwidth}
			\rotatebox{90}{\small{Noise roughness}}
		\end{minipage}
		\begin{minipage}[t]{0.03\textwidth}
			\rotatebox{90}{ $\beta = 1.5$}
		\end{minipage}
		\includegraphics[scale = 1]{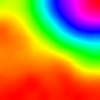}
		\includegraphics[scale = 1]{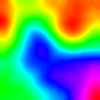}
		\includegraphics[scale = 1]{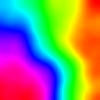}
		\includegraphics[scale = 1]{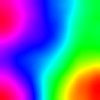}
		\includegraphics[scale = 1]{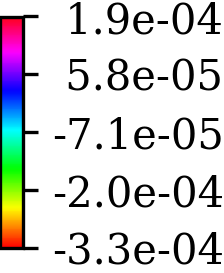}
	\end{subfigure}
	
	\caption{\textbf{Experiment 4.} Evolution of the relative $L^2$-error (left) and the $L^\infty$-error (right) between the reference solution and the perturbed solution. The last row shows snapshots of the filtered Gaussian random field used.} 
	\label{fig:experiment4_errors}
\end{figure}

%
%
\section{Conclusion and discussion}\label{sec:discussion}
In this work, we studied determining modes for a class of two-species reaction-diffusion systems arising in models of pattern formation. Under assumptions \ref{P1}-\ref{P6} and the additional damping condition \eqref{eq:extra_assumption}, we proved in \Cref{thm:lambda_main} that the long-time dynamics are determined by finitely many Fourier modes. The result applies in particular to several classical systems, including the Gray--Scott, and Glycolysis models.

Using a data assimilation algorithm with linear feedbacks for the Gray--Scott model  we illustrate the determining modes condition \Cref{thm:lambda_main} and show synchronization between the reference and assimilated dynamics once a sufficient number of low modes is observed. Although the theoretical estimate obtained in \Cref{thm:lambda_main} is highly conservative, synchronization is observed numerically for a much smaller number of modes. Such large gaps between theoretical and numerical results are typical in determining modes theory, where analytical bounds provide sufficient conditions but generally overestimate the effective number of degrees of freedom needed in computations.

The present work also complements the recent work \cite{randrianasolo2025discrete}, which demonstrates that one can reconstruct the full Gray--Scott solution from coarse, cell-averaged data via the Azouani--Olson--Titi (AOT) data assimilation algorithm, as given by  
\begin{equation}\label{eq:nudgedgs_FV}
	\left\{
	\begin{aligned}
		\partial_t{\tu} &= d_u\Delta \tu  -\tu\tv^2 + F(1-\tu) + \mu_u\big( \mI_H u -\mI_H \tu\big),
		\\
		\partial_t{\tv} &= d_v\Delta \tv  + \tu\tv^2 - (F+k)\tv + \mu_v\big(\mI_H v - \mI_H \tv\big).
	\end{aligned}
	\right.
\end{equation}
Mathematically, the result in~\cite[Theorem~4.5]{randrianasolo2025discrete} demonstrates \textit{synchronization}: the nudged state $(\tilde{u}, \tilde{v})$ converges exponentially in $L^2(D)$ to the true state $(u,v)$, provided that the observation operator $\mI_H$ resolves enough spatial structure, that is the observation resolution $H$ is sufficiently small relative to diffusion and the feedback gain $\mu_u,\mu_v$. In other words, once the low-resolution component, captured by $\mI_H$, of the error decays, the high-frequency components automatically follow due to diffusion and nonlinear coupling. 
This result motivates the question we address here,  
\begin{quote}
	\textit{How many modes of the Gray--Scott system or another model of pattern formation  are sufficient to determine the full dynamics?}
\end{quote}
The present work on determining modes answers this question analytically at the continuous PDE level, without referring to discrete data or to a specific data assimilation algorithm except in the numerical experiments.

The structure of synchronization in the AOT system is, in fact, 
mathematically analogous to that of the determining modes property proved here in \cref{thm:lambda_main}. Indeed, the determining modes result can be expressed as
\[
\limsup_{t\to\infty}  \Big(\Vert (g_1-\tg_1)(t)\Vert_{L^2}+ \Vert (g_2-\tg_2)(t)\Vert_{L^2}\Big) = 0, \mbox{ and } \lim_{t\to\infty}\Big(\Vert P_{M}(u - \tu)(t)\Vert_{L^2} +  \Vert P_{N}(v - \tv)(t)\Vert_{L^2}\Big)= 0
\]
imply
\begin{equation*}
	\lim_{t \to \infty} \big(
	\Vert Q_M (u - \tilde{u}) \Vert_{L^2} + \Vert Q_N (v - \tilde{v}) \Vert_{L^2}
	\big) = 0.
\end{equation*} 

In both settings, the low modes (or coarse observations) act as \textit{determining observables}. The high modes (fine details) are asymptotically determined to those determining observables. The constants and inequalities connecting the feedback parameters $(\mu_u,\mu_v,H)$ in the data assimilation framework correspond analytically to the determining mode inequality
\begin{equation*}
	 \frac{\min\{\lambda_{M},\lambda_{N}\}}{36\lambda_1}
	>  \Gr +   \gamma^2 \bigg[d^2\Big(1 + \frac{d\lambda_1 \sqrt{2}}{c_1c_2}\Big)^4  + \frac{2d^5\lambda_1^3}{c_3^3}(1 + \alpha)^2\bigg]\Gr^4.
\end{equation*}

Hence, the main result \cref{thm:lambda_main} of the present paper provides the analytical foundation for why data assimilation via coarse measurements can succeed: the feedback in the AOT algorithm acts precisely on the determining subspace that governs the full system dynamics.
This correspondence highlights the central idea of the present work, notably that complex reaction-diffusion systems, such as the Gray--Scott, admit a finite-dimensional description capturing their essential \textit{reduced dynamics}.

One limitation of the present result is that the nonlinearities are restricted to polynomial reaction terms satisfying assumptions \ref{P1}-\ref{P6}. In particular, the approach does not directly apply to activator-inhibitor systems with singular kinetics, such as the classical Gierer--Meinhardt model. While the stochastic Gierer--Meinhardt system has been studied in other contexts, see Hausenblas and Panda~\cite{hausenblas2022the}, the singular nonlinear term prevents a direct application of the present determining modes approach. The issue starts from \Cref{lem:form}, where we extract the feedbacks $P_M\xi$ and $P_N\eta$ from the nonlinear terms. Extending determining modes theory to systems with singular kinetics remains an interesting direction for future research.

%
%
\section*{Declarations}
\subsection*{Acknowledgments} 
The authors would like to thank anonymous referees for the careful reading of the manuscript and for the valuable comments and suggestions, which greatly improved the quality and presentation of this work. 

\subsection*{Funding}
This work was supported by the Austrian Science Fund (FWF) \href{https://www.fwf.ac.at/en/research-radar/10.55776/P34681}{10.55776/P44681}. The second author was supported by the Deutsche Forschungsgemeinschaft (DFG, German Research Foundation) -- Project-ID \href{https://gepris.dfg.de/gepris/projekt/317210226}{317210226 - SFB 1283}.

\subsection*{Conflicts of Interest} The authors have no competing interests to declare that are relevant to the content of this article.

\subsection*{Reproducibility} 
The numerical experiments were implemented in Python using, in particular, \texttt{scipy.fft} for the spectral cosine discretization, while the figures were post-processed using Python notebooks.
The code and synthetic data used in this paper are available from the corresponding author upon reasonable request.

%
%
\bibliographystyle{abbrvurl}
\bibliography{determiningModesBitex4}
\end{document}